\documentclass[a4paper]{amsart}
\usepackage{amssymb} 
\usepackage{amsmath, graphicx} 
\usepackage{comment}
\usepackage{color}
\usepackage{array, booktabs, comment, mathrsfs, enumerate, multirow, soul}
\usepackage[neverdecrease]{paralist}
\usepackage[h]{esvect}
\usepackage{tikz-cd}
\usepackage{todonotes}
\usepackage{microtype}
\usepackage{url}
\usepackage[colorlinks]{hyperref}
\hypersetup{citecolor=blue}

\DeclareMathOperator{\PGL}{PGL}
\DeclareMathOperator{\PSL}{PSL}
\DeclareMathOperator{\Art}{Art}
\DeclareMathOperator{\Aut}{Aut}

\DeclareMathOperator{\N}{N}
\DeclareMathOperator{\Cl}{Cl}

\DeclareMathOperator{\GL}{GL}
\DeclareMathOperator{\Jac}{Jac}

\DeclareMathOperator{\Gal}{Gal}

\DeclareMathOperator{\Emb}{Emb}
\DeclareMathOperator{\End}{End}

\DeclareMathOperator{\ord}{ord}

\DeclareMathOperator{\Spec}{Spec}
\DeclareMathOperator{\Spf}{Spf}
\DeclareMathOperator{\GSpin}{GSpin}

\DeclareMathOperator{\Stab}{Stab}
\DeclareMathOperator{\Star}{Star}
\DeclareMathOperator{\Nrd}{Nrd}

\DeclareMathOperator{\Ind}{Ind}

\DeclareMathOperator{\Tr}{Tr}

\newcommand{\A}{{\mathbf A}}
\newcommand{\Q}{{\mathbf Q}}
\newcommand{\Z}{{\mathbf Z}}
\newcommand{\C}{{\mathbf C}}
\newcommand{\R}{{\mathbf R}}
\newcommand{\F}{{\mathbf F}}
\newcommand{\Ha}{{\mathbf H}}

\newcommand{\CO}{\mathcal{O}}

\newcommand{\gc}{{\mathfrak{c}}}

\newcommand{\gm}{\mathfrak{m}}

\newcommand{\gO}{\mathfrak{O}}
\newcommand{\gP}{\mathfrak{P}}
\newcommand{\gp}{\mathfrak{p}}
\newcommand{\gq}{\mathfrak{q}}

\newcommand{\gN}{\mathfrak{N}}

\newcommand{\Qbar}{\overline{\Q}}

\newcommand{\T}{\mathbf{T}}

\theoremstyle{plain}
\newtheorem*{thmA}{Theorem A}
\newtheorem*{thmB}{Theorem B}
\newtheorem*{thmC}{Theorem C}
\newtheorem{thm}{Theorem}[section]
\newtheorem{prop}[thm]{Proposition}

\newtheorem{cor}[thm]{Corollary}

\newtheorem{lem}[thm]{Lemma}

\newtheorem{rem}[thm]{Remark}

\title[A hyperelliptic Shimura quotient of genus $16$]{An intriguing hyperelliptic Shimura curve quotient of genus $16$} 

\begin{document}

\author{Lassina Demb\'el\'e}
\address{Department of Mathematics, University of Luxembourg, Esch-sur-Alzette L-4364, Luxembourg}
\email{lassina.dembele@gmail.com}
\thanks{The author was supported in part by EPSRC Grants EP/J002658/1 and EP/L025302/1, a Visiting grant from the Max-Planck Institute for Mathematics, 
a Simons Collaboration Grant (550029), and Luxembourg FNR Fund PRIDE/GPS/12246620.}
\date{\today}


\maketitle

\begin{abstract} 
Let $F$ be the maximal totally real subfield of $\Q(\zeta_{32})$, the cyclotomic field of $32$nd roots of unity. Let $D$ be the
quaternion algebra over $F$ ramified exactly at the unique prime above $2$ and 7 of the real places of $F$. Let $\CO$ be a
maximal order in $D$, and $X_0^D(1)$ the Shimura curve attached to $\CO$. Let $C = X_0^D(1)/\langle w_D \rangle$, where
$w_D$ is the unique Atkin-Lehner involution on $X_0^D(1)$. We show that the curve $C$ has several striking features. First,
it is a hyperelliptic curve of genus $16$, whose hyperelliptic involution is exceptional. Second, there are $34$ Weierstrass points 
on $C$, and exactly half of these points are CM points; they are defined over the Hilbert class field of the unique CM extension $E/F$
of class number $17$ contained in $\Q(\zeta_{64})$, the cyclotomic field of $64$th roots of unity. Third, the normal closure of the field 
of $2$-torsion of the Jacobian of $C$ is the Harbater field $N$, the unique Galois number field $N/\Q$ unramified outside $2$
and $\infty$, with Galois group $\Gal(N/\Q)\simeq F_{17} = \Z/17\Z \rtimes (\Z/17\Z)^\times$. In fact, the Jacobian $\Jac(X_0^D(1))$ has
the remarkable property that each of its simple factors has a $2$-torsion field whose normal closure is the field $N$. Finally, and
perhaps the most striking fact about $C$, is that it is also hyperelliptic over $\Q$. 
\end{abstract}

\section{\bf Introduction}

Let $F$ be the maximal totally real subfield of $\Q(\zeta_{32})$, the cyclotomic field of $32$nd roots of unity. 
Recall that $2$ is totally ramified in $F$, and let $\gp$ be the unique prime above it. Let $D$ be the
quaternion algebra defined over $F$ ramified exactly at $\gp$ and 7 of the real places of $F$. Let $\CO$ be a
maximal order in $D$, and $X_0^D(1)$ the Shimura curve attached to $\CO$. Let $C = X_0^D(1)/\langle w_D \rangle$, where
$w_D$ is the unique Atkin-Lehner involution on $X_0^D(1)$. Let $\Jac(X_0^D(1))$ and $\Jac(C)$ be the Jacobians of $X_0^D(1)$
and $\Jac(C)$, respectively. In this note, we show that $C$ has several striking properties. First, we prove the following theorem 
(Theorem~\ref{thm:C-is-hyperelliptic-over-Q}).

\begin{thmA}\label{thm:thmA} 
The curve $C$ is a hyperelliptic curve of genus $16$ defined over $\Q$.
\end{thmA}
We first show that $C$ is hyperelliptic over $F$ (Theorem~\ref{thm:C-is-hyperelliptic}), then we apply a descent argument from~\cite{sv16} 
to show that both the curve and the hyperelliptic involution are defined over $\Q$. For the first part, we simply count the number of Weierstrass points on $C$.
This count yields that $C$ has $34$ Weierstrass points, the maximum number for a hyperelliptic curve of genus $16$ by the Weierstrass gap theorem. 
Half of those Weierstrass points are CM points defined over the Hilbert class field of the  unique CM extension $E/F$ of class number $17$ contained in 
$\Q(\zeta_{64})$, the cyclotomic field of $64$th roots of unity. 

To show that $C$ is in fact defined over $\Q$, we determine the automorphism group $\Aut(C)$ of $C$ as a curve over $F$. We do this by exploiting the 
\v{C}erednik-Drinfel'd $2$-adic uniformisation of $X_0^D(1)$ and the fact that the automorphism group of a stable curve injects into an {\it admissible} subgroup 
of the automorphism group of its dual graph (see~\cite{dm69} and \S \ref{subsec:automorphism-group} for the definition of admissibility). A careful study of the 
dual graph of the stable model of $C$ over the completion of $F$ at $\gp$ then yields that $\Aut(C) = \Z/2\Z$.  As a result, we get that the {\it only} non-trivial 
automorphism of $C$ is the hyper\-el\-lip\-tic involution, which in this case must be exceptional since the curve $C$ is obtained as the quotient of $X_0^D(1)$ 
by the unique Atkin-Lehner involution $w_D$. 


Our second result concerns the field of $2$-torsion of $\Jac(C)$. It is known that $17$ is the smallest odd integer which can occur as the 
degree of a number field $K/\Q$ for which $2$ is the only finite prime which ramifies. That there is no such integer less than $17$
follows from work of Jones~\cite{jon10}. On the other hand, Harbater~\cite{har94} proves that there is a unique Galois number field $N/\Q$
unramified outside $2$ and $\infty$, with Galois group $\Gal(N/\Q) \simeq F_{17} = \Z/17\Z \rtimes (\Z/17\Z)^\times$. So, the fixed field of
the Sylow $2$-subgroup of $F_{17}$ is a number field of degree $17$ in which $2$ is the only ramified finite prime. Noam Elkies provides a degree $17$ 
polynomial whose splitting field is $N$. The computation which led to that polynomial stemmed from a discussion on \verb|mathoverflow.net|~\cite{er15}
initiated by Jeremy Rouse. In the context of that discussion, it is natural to ask whether there is a curve defined over $\Q$, with good reduction 
away from $2$, whose field of $2$-torsion is the Harbater field $N$. The following theorem provides an affirmative answer to that question 
(Theorem~\ref{thm:harbater-field}).
\begin{thmB} The field  of $2$-torsion of $\Jac(C)$ is the Harbater field $N$.
\end{thmB}

The fact that the Harbater field can be realised as the field of $2$-torsion of a {\it hyperelliptic} curve of rather large genus, with good reduction outside $2$,
seems rather remarkable to us. For that reason, we think that it would be very interesting to find a defining equation for $C$ over $\Q$. This is a question
of independent interest that we hope to consider in the future. 

In fact, we prove a slightly stronger result than Theorem~B. Namely, up to isogeny, the Jacobian $\Jac(X_0^D(1))$ decomposes as the product of four
abelian varieties of dimension $4$ and one of dimension $24$. We give two different proofs of the following (Theorem~\ref{thm:2-torsion-fields}). 

\begin{thmC} Let $A$ be a simple factor of $\Jac(X_0^D(1))$. Then the normal closure of the field of $2$-torsion of $A$
is the Harbater field $N$.
\end{thmC}

The second proof of Theorem~C uses congruences. Namely, let $S_2^D(1)$ be the space of automorphic forms of level $(1)$ and weight $2$ 
over the quaternion algebra $D$, and $\T$ be the Hecke algebra acting on $S_2^D(1)$. We show that there are two congruence classes modulo 
$2$ among the newforms in $S_2^D(1)$, whose associated mod $2$ residual Galois representations have the same  image $D_{17}$. These two 
congruence classes are permuted by $\Gal(F/\Q)$. As a result, we get that the normal closure of the field of $2$-torsion of every simple factor of 
$\Jac(X_0^D(1))$ is the Harbater field $N$. Interestingly, the existence of these two {\it distinct} congruence classes modulo $2$ turns out to have
the following amusing consequence: The connectedness of $\Spec(\T)$, which is obtained by an argument \`a la Mazur~\cite[Proposition 10.6]{maz77}, 
{\it cannot} arise from a single congruence modulo $2$. In other words, the existence of the Harbater field as the normal closure of the field of $2$-torsion 
of $\Jac(X_0^D(1))$ is an obstruction to the connectedness of $\Spec(\T)$ being achieved via a unique congruence modulo $2$. This is due to the
tautological reason that the semi-direct product $F_{17} = D_{17}\rtimes \Z/8\Z$ is {\it non-split}. In fact, we show that the connectedness of $\Spec(\T)$ 
is given by two different congruences modulo $3$ and $5$. 

Our initial interest in the curve $X_0^D(1)$ stems from a conjecture of Benedict H. Gross which states that, for any prime $p$, 
there is a non-solvable number field $K/\Q$ ramified at $p$ (and possibly at $\infty$) only. In~\cite{dem09}, we proved that conjecture for $p = 2$ by using 
Hilbert modular forms of level $(1)$ and weight $2$ over $F$. Theorem C implies that none of the simple factors of $\Jac(X_0^D(1))$ has a $2$-torsion 
field that can be used to provide an affirmative answer to the Gross conjecture for number fields given that $N$ is solvable. Amusingly, it turns out that the 
simple factors of $\Jac(X_0^D(1))$ are more interesting in relation to other conjectures of Gross~\cite{gro16} which concern modularity of abelian varieties 
not of $\GL_2$-type. Indeed, functorially, these simple factors are related to abelian varieties defined over $\Q$ with {\it small} or even trivial endomorphism 
rings, but which acquire extra endomorphisms over $F$, as we explain later (see also~\cite{cd17}).

The outline of the paper is as follows. In Section~\ref{sec:weierstrass-points}, we recall the necessary background on Weierstrass points and
hyperellipticity. In Section~\ref{sec:fuchsian-groups}, we recall the necessary background on arithmetic groups in quaternion algebras, and
compliment this by discussing optimal embeddings into maximal arithmetic Fuchsian groups. In Section~\ref{sec:shimura-curves}, we review
the theory of Shimura curves, especially their $p$-adic uniformisation. Finally, in Sections~\ref{sec:gross-curve}
and~\ref{sec:2-torsion-fields}, we discuss our example, its Jacobian and the connection of their $2$-torsion fields with the Harbater field.

\medskip
\noindent
{\bf Acknowledgements.} I would like to thank Frank Calegari for several helpful email exchanges; Vladimir Dokchitser and C\'eline
Maistret for some useful discussion; and Jeroen Sijsling for carefully reading an earlier draft of this work. I would also like to give a special
thanks to John Voight as this note owes a lot to the lengthy discussions I had with him on this topic. I learned of the discussion about 
the Harbater field between Jeremy Rouse and Noam D. Elkies from David P.~Roberts who pointed us to the \verb|mathoverflow.net| post 
related to this. So, I would like to thank him for this. During the course of this project, I stayed at the following institutions: Dartmouth 
College, King's College London, the Max-Planck Institute for Mathematics in Bonn, and the University of Barcelona; I would like to thank 
them for their generous hospitality. I also thank the referees for many helpful suggestions. Finally, as alluded to earlier, this note originated 
with a question of Benedict Gross. So I would like to thank him for this, and for his constant encouragement.

\section{\bf Background on Weierstrass points}\label{sec:weierstrass-points}
Throughout this section, $X$ is a smooth projective curve of genus $g \ge 2$ defined over a field $k$ of characteristic $0$,
with algebraic closure $\overline{k}$.

\subsection{Definition and properties}
Let $P$ be a point on $X$. We say that $P$ is a {\it Weierstrass point} if there exists a differential form $\omega \in H^0(X, \Omega^1_X)$
such that $\ord_P(\omega) \ge g$. We let $\mathscr{W}$ be the set of all Weierstrass points on $X(\overline{k})$. 
Alternatively, one can describe $\mathscr{W}$ as follows. Let $D$ be a divisor on $X$, and $\mathscr{L}(D)$ the Riemann-Roch space
associated to $D$, i.e.
$$\mathscr{L}(D) := \{f \in k(X)^\times: \mathrm{div}(f) + D \ge 0\}\cup \{0\}.$$ 
By the Riemann-Roch Theorem, $\mathscr{L}(D)$ is finite dimensional, and we let $\ell(D)$ be its dimension. 

\begin{prop} Let $P$ be a point on $X$. Then, $P \in \mathscr{W}$ if and only if $\ell(gP) \ge 2$. 
\end{prop}

\begin{proof} This is a consequence of the Riemann-Roch Theorem~\cite[\S A.4]{hs00}.
\end{proof}

The {\it gap sequence} associated to a Weierstrass point $P$ is the set
$$G(P) := \{ n \in \Z_{\ge 0} : \ell(nP) = \ell((n-1)P)\}.$$
The {\it weight} of the Weierstrass point $P$ is defined by
$$w(P) := \left(\sum_{n \in G(P)} n\right) - \frac{g(g+1)}{2}.$$

\begin{thm} Let $P$ be a point on $X$. Then $P$ is a Weierstrass point if and only if $w(P) \ge 1$, and 
$\sum w(P)P$ belongs to the complete linear system 
$$\left|\frac{g(g+1)}{2} K_X\right|,$$
where $K_X$ is a canonical divisor on $X$. 
In particular, we have that 
$$\sum_{P \in \mathscr{W}}w(P) = g(g^2 - 1).$$
\end{thm}

\begin{proof} See~\cite[\S III.5]{fk92} or~\cite[Exercise A.4.14]{hs00}.
\end{proof}

\subsection{Hyperellipticity} We recall that $X$ is a {\it hyperelliptic} curve if there is a degree
$2$ map $\phi:\,X \to \mathbf{P}^1$ defined over $\overline{k}$. In that case, $\phi$ is unique (up to automorphisms of 
$\mathbf{P}^1$).  The map $\phi$ induces a degree $2$ extension $\overline{k}(X)/\overline{k}(\mathbf{P}^1)$, which is Galois since 
$\mathrm{char}(k) = 0$. So, this gives rise to a map $\iota:\, X \to X$ called the {\it hyperelliptic involution}. We say 
$X$ is hyperelliptic over $k$ if $\phi$ is defined over $k$. The following is a well-known classical result. 

\begin{prop}\label{prop:hyperellipticity} Let $X$ be a curve of genus $g\ge 2$ defined over a field $k$ of
characteristic $0$, and $\mathscr{W}$ the set of Weierstrass points of $X(\overline{k})$. Then, we have
$$ 2g + 2 \le \#\mathscr{W} \le g^3 - g.$$
Furthermore, $X$ is hyperelliptic if and only if $\#\mathscr{W} = 2g + 2$. In that case, the branch points are the Weierstrass
points.
\end{prop}

\begin{proof}
See~\cite[\S III.5]{fk92} or~\cite[Exercise A.4.14]{hs00}.
\end{proof}

\subsection{Galois action}\label{subsec:galois-action}

Let $\mathscr{W}$ be the set of all Weierstrass points over $X(\overline{k})$, then $\mathscr{W}$ is preserved by the action of 
$\Gal(\overline{k}/k)$. In particular, when $X$ is a hyperelliptic curve, this action factors through the symmetric group $S_{2g + 2}$.

\section{\bf Arithmetic Fuchsian groups}\label{sec:fuchsian-groups}

From now on, $F$ is a totally real number field of degree $g$. We denote the real embeddings of $F$ by $v_1,\,\ldots,\,v_g$. 
We let $\CO_F$ be the ring of integers of $F$, and $\CO_F^{\times +}$ the group of totally positive units in $\CO_F$. 
We let $D$ be a quaternion algebra defined over $F$, and fix a maximal order $\CO$
in $D$. Let $v$ be a place of $F$, and $F_v$ the completion of $F$ at $v$. We recall that $D$ is said to be ramified at $v$ if 
$D_v = D \otimes F_v$ is a division quaternion algebra. We let $S_\infty$ (resp. $S_f$) be the set of archimedian places 
(resp. finite places) where $D$ is ramified; and set $S = S_\infty \cup S_f$. We let $r = \#S_f$. 

\subsection{Fuchsian groups} From now on, we assume that $D$ is ramified at all but one archimedian places; namely, that 
$S_\infty =\{v_2,\,\ldots,v_g\}$. This means that, we have $D \otimes \R \simeq \mathrm{M}_2(\R) \times \Ha^{g-1}$, where 
$\Ha$ is the Hamilton quaternion algebra over $\R$. We let $j_1:\,D\otimes_{v_1}\R \to \mathrm{M}_2(\R)$ be the 
projection onto the factor corresponding to $v_1$. We will also denote the map induced on the unit groups by 
$j_1:\, (D\otimes_{v_1}\R)^\times \to \GL_2(\R)$. For the definition of the reduced norm $\Nrd: D \to F$ below, 
we refer to \cite[Chap. I, \S 1]{vig80} or \cite[\S 3.3]{voi18}. We let
 \begin{align*}
\CO^1 &:= \left\{x \in \CO: \Nrd(x) = 1\right\};\\
\CO^\times &:= \left\{x \in \CO: \Nrd(x) \in \CO_F^\times \right\};\\
\CO_{+}^\times &:= \left\{x \in \CO: \Nrd(x) \in \CO_F^{\times +} \right\}.
\end{align*}
We recall that the {\it normaliser} of $\CO$ inside $D$ is defined by
$$N_D(\CO) := \left\{ x \in D^\times: \, x\CO = \CO x \right\}.$$
We set
$$N_D(\CO)_{+} := \left\{ x \in N_D(\CO) : \Nrd(x) \in F_{+}^\times \right\}.$$
We let $\Gamma^1$ (resp. $\Gamma$ and $\Gamma_\CO$) be the image of $\CO^1$ (resp. $\CO_{+}^\times$ and $N_D(\CO)_{+}$) 
in $\PGL_2^+(\R) := \GL_2^+(\R)/\R^\times$ via $j_1$, where $$\GL_2^+(\R) := \left\{\gamma \in \GL_2(\R): \det(\gamma) > 0 \right\}.$$
We will also use the same notation to identify these groups with their respective images in $D^\times/F^\times$.
We recall that $\Gamma^1$ is an {\it arithmetic Fuchsian group}, i.e. a discrete subgroup of $\PSL_2(\R)$. 
The {\it commensurability class} of $\Gamma^1$, consists of all the subgroups $\Gamma' \subset \PGL_2^{+}(\R)$ that are
commensurable with $\Gamma^1$, i.e. such that $\Gamma' \cap \Gamma^1$ has finite index in both $\Gamma'$ and $\Gamma^1$.
Any Fuchsian group that is commensurable to an arithmetic Fuchsian group is itself arithmetic. So, the commensurability class of
$\Gamma^1$ is independent of the embedding $j_1$. We define it simply as the commensurability class of $\CO^1$ in 
$D^\times/F^\times$, and denote it by $\mathscr{C}(D)$. In $\mathscr{C}(D)$, one is particularly interested in those groups $\Gamma'$
with minimal covolume. Borel~\cite{bor81} shows that, up to conjugacy, there are finite many such groups, and gives their covolume
purely in terms of the number theoretic data used in defining them. These groups are called {\it maximal arithmetic Fuchsian groups}, 
and are the main objects of interest to us in this section.

\begin{thm}[Borel~\cite{bor81}]\label{thm:borel-volume}
Every maximal arithmetic Fuchsian group in $\mathscr{C}(D)$ is of the form $\Gamma_\CO$, where $\CO$ is a maximal order in $D$. In that case, the covolume
of $\Gamma_\CO$ is given by $$\mathrm{Vol}(\Gamma_\CO\backslash \mathfrak{H}) = \frac{8 \pi D_F^{3/2}\zeta_F(2)}{(4\pi^2)^g[H : {F^\times}^2]} \prod_{\gq \in S_f}(\N\gq - 1),$$
where $H = \{\Nrd(x) : x \in N_D(\CO)_{+}\}$. In particular, it depends only on $F$ and $S_f$. 
\end{thm}

\begin{proof} See Borel~\cite[\S 8.4]{bor81}. 
\end{proof}

\subsection{The Atkin-Lehner group}\label{subsec:atkin-lehner-gp}
We define the {\it Atkin-Lehner group}
$$W := N_D(\CO)/F^\times \CO^\times.$$
By the Skolem-Noether Theorem~\cite[Chap. II, Th\'eor\`eme 2.1]{vig80}, $W$ can be identified with the group of automorphisms of $\CO$. It is generated by the
classes $[u] \in W$ such that $(u)$ is a principal two-sided ideal whose norm is supported at the prime ideals in $S_f$. By the Hasse-Schilling-Maass 
Theorem~\cite[Chap. III, Th\'eor\`eme 5.7]{vig80}, $W$ is a {\it finite} elementary abelian $2$-group. So, there is a positive integer $r$ such that
$$W \simeq (\Z/2\Z)^r.$$
We define the {\it positive Atkin-Lehner groups}
\begin{align*}
W_{+} &:= N_D(\CO)_+/F^\times \CO_{+}^\times,\\
W^1 &:= N_D(\CO)/F^\times \CO^1.
\end{align*}
There is a split exact sequence
$$1 \to \CO_F^{\times +}/(\CO_F^\times)^2 \to W^1 \to W_{+} \to 1,$$
which gives an isomorphism $$W^1 \simeq \CO_F^{\times +}/(\CO_F^\times)^2 \times W_{+} \simeq (\Z/2\Z)^s,$$
where $s \le (n-1) + r$. The rank $s$ of $W^1$ can be determined from the Dirichlet unit theorem and the fact that
the image of $W_{+}$ inside $W$ is generated by those principal two-sided ideals whose norms are totally positive 
and supported at $S_f$.

\subsection{Optimal embeddings}
Let $E/F$ be a CM extension, i.e. a totally imaginary quadratic extension. By~\cite[Chap. III, Th\'eor\`eme 3.8]{vig80}, $E$ embeds into $D$ if and only if,
every finite place $v \in S_f$ is ramified or inert in $E$. The following theorem will be very useful for us.  

\begin{thm}\label{thm:chinburg-friedman}
Let $E/F$ be a {\rm CM} extension, and $\sigma:\, E \hookrightarrow D$ an embedding. Let $\alpha \in E \setminus F$, and 
$\mathrm{disc}(\alpha) = \Tr_{E/F}(\alpha)^2 - 4 \N_{E/F}(\alpha)$. Then, up to conjugation, $\sigma(\alpha) \in N_D(\CO)_{+}$ 
if and only if $\mathrm{disc}(\alpha)/\N_{E/F}(\alpha) \in \CO_F$, and $\N_{E/F}(\alpha) \in F_{+}^\times$ is supported at $S_f$ 
modulo squares. 
\end{thm}

\begin{proof} This follows from Chinburg-Friedman~\cite[Lemma 4.3]{cf99} (see also Maclachlan~\cite[Theorem 3.1]{mac06}). 
\end{proof}

Let $E/F$ be a CM extension, and $\gO$ an $\CO_F$-order in $E$. 
An {\it optimal embedding} of $\gO$ in $\CO$ is a homomorphism $\iota:\, E \hookrightarrow D$ such that
$\iota(\gO) = \iota(E) \cap \CO$. We denote the set of optimal embeddings of $\gO$ into
$\CO$ by $\Emb(\gO, \CO)$. We fix an embedding $E \hookrightarrow D$. Then, by the Skolem-Noether Theorem, 
every embedding of $E$ into $D$ is of the form $(x \mapsto \alpha x \alpha^{-1})$ for some $\alpha \in D^\times$. 
So, we can identify $\Emb(\gO, \CO)$ with the coset space $E^\times\backslash\mathscr{E}(\gO, \CO)$ where
\begin{align*}
\mathscr{E}(\gO, \CO) &:= \left\{ \alpha \in D^\times : \alpha E \alpha^{-1} \cap \CO = \gO \right\}\\
&\,\,= \left\{ \alpha \in D^\times : E \cap \alpha^{-1} \CO \alpha = \alpha^{-1}\gO\alpha \right\}.
\end{align*}
Conjugation induces a right action of $N_D(\CO)/F^\times$ on $\Emb(\gO, \CO)$.
For any subgroup $\Gamma^1 \subset \Gamma \subset N_D(\CO)/F^\times$, we let 
$\Emb(\gO, \CO; \Gamma)$ be the set of $\Gamma$-conjugacy classes of optimal embeddings. 
Similarly, if $\CO^1 \subset G \subset N_D(\CO)$, we let
$\Emb(\gO, \CO; G) := \Emb(\gO, \CO; \overline{G})$, where $\overline{G}$ is the image of $G$ 
in $N_D(\CO)/F^\times$. The set $\Emb(\gO, \CO; \Gamma)$ is {\it finite} since $\Gamma$ has finite
index in $N_D(\CO)/F^\times$. The cardinality $m(\gO, \CO; \Gamma)$ of this set is called the embedding number of
$\gO$ into $\CO$, with respect to $\Gamma$; or simply the embedding number of $\gO$ into $\CO$ when $\Gamma = \CO^\times$. 
There are formulae for $m(\gO, \CO; \CO^\times)$, see for example~\cite[Chap. II, \S 3 and Chap. III, \S 5]{vig80} or~\cite[\S 30]{voi18}. 
The following lemma can be used to get  $m(\gO, \CO; G)$ for any subgroup $\CO^1 \subset G \subset \CO^\times$. 

\begin{lem}\label{lem:mO1} Let $\CO^1 \subset G \subset \CO^\times$ be a subgroup. 
Then we have
$$m(\gO, \CO; G) = m(\gO, \CO; \CO^\times) [\Nrd(\CO^\times) : \Nrd(G)\N_{E/F}(\gO^\times)].$$
\end{lem}

\begin{proof} See Voight~\cite[Lemma 30.3.14]{voi18}. (We note that the statement in Vign\'eras \cite[Chap. III, Corollaire 5.13]{vig80} is
only correct with the inclusion $G \subset N_D(\CO)$ replaced by $G \subset \CO^\times$.)  
\end{proof}

Here we are interested in the case when  $\CO_{+}^\times \subset G \subset N_D(\CO)_{+}$. In particular,
we want $\Emb(\gO, \CO; N_D(\CO)_{+})$ when $\CO$ is a maximal order in $D$.

\begin{lem}\label{lem:mO2} Let $\CO_{+}^\times \subset G \subset N_D(\CO)_{+}$ be a subgroup. 
Then we have
$$m(\gO, \CO; \CO_{+}^\times) = m(\gO, \CO; G) \left[ \Nrd(G) : \Nrd(G) \cap \N_{E/F}(E^\times)\CO_F^{\times +}\right].$$
\end{lem}

\begin{proof} There is a natural surjection
$$E^\times\backslash\mathscr{E}(\gO, \CO)/\CO_{+}^\times \to E^\times\backslash\mathscr{E}(\gO, \CO)/G.$$
To prove the lemma, we need to understand the fibres of this map. For $\alpha \in \mathscr{E}(\gO, \CO)$, the fibre of $E^\times \alpha G$ is 
$$T := E^\times\backslash E^\times \alpha G/\CO_{+}^\times \simeq \left(\alpha E^\times \alpha^{-1} \cap G\right)\backslash G /\CO_{+}^\times.$$
It is enough to show that the cardinality of $T$ is independent of $\alpha$. To see this, we recall that the reduced norm $\Nrd:\, D_{+}^\times \to F_{+}^\times$
induces a map
$$\phi:\,\left(\alpha E^\times \alpha^{-1} \cap G\right)\backslash G /\CO_{+}^\times \to \Nrd(G)/\Nrd(G)\cap \N_{E/F}(E^\times)\CO_F^{\times +},$$
which is a bijection since $\ker(\phi) = \CO^1 \subset \CO_{+}^\times \subset G \subset N_D(\CO)_{+}$. 

Alternatively, we can observe that $\CO_{+}^\times$ is a normal subgroup of $G$. So, we can identify 
$\left(\alpha E^\times \alpha^{-1} \cap G\right)\backslash G /\CO_{+}^\times$ with a subgroup of $W_{+}$.
This means that $\# T$ divides $\#W_{+}$, and is always a power of $2$. 
\end{proof}

Let $\widehat{\CO} := \CO \otimes \widehat{\Z} = \prod_{v<\infty}\CO_v$ and $\widehat{D} := D\otimes\widehat{\Q}$, where $\widehat{\Z}$ and $\widehat{\Q}$ are the finite ad\`eles of $\Z$
and $\Q$, respectively. For every finite place $v$, let $\CO_v^\times \subset G_v \subset N_{D_v}(\CO_v)$ be a subgroup, and $\widehat{G} := \prod_{v<\infty}G_v$. 
We would like to understand the global embedding numbers of the group $\widehat{G}$, or $G := \widehat{G} \cap D_{+}^\times$. Since $D$ satisfies the 
Eichler condition, we have  $D_{+}^\times \backslash \widehat{D}^\times/\widehat{\CO}^\times \simeq \Cl_F^{+}$, where $\Cl_F^{+}$ is the narrow class group of $F$. 
Let $h = \#\Cl_F^{+}$ be the narrow class number of $F$, and $$\widehat{D}^\times  = \coprod_{i = 1}^h D_{+}^\times g_i \widehat{\CO}^\times,$$
where $g_i \in \widehat{D}^\times$, $i = 1,\,\ldots,\,h$, and $g_1 = 1$. Then, for each $i$, $\CO_i := g_i \widehat{\CO}g_i^{-1} \cap D$ is a maximal order, 
and $N_D(\CO_i) = g_i N_{\widehat{D}}(\widehat{\CO})g_i^{-1} \cap D$. 
Letting $G_i := g_i \widehat{G} g_i^{-1} \cap D_{+}^\times$, we have $(\CO_i)_{+}^\times \subset G_i \subset N_D(\CO_i)_{+}$. 

For $\widehat{G} = \widehat{\CO}^\times$, there are formulae for global optimal embeddings numbers (see~\cite[Chap. III, \S 5]{vig80} or~\cite[\S 30]{voi18}).
For $\widehat{\CO}^\times \subset \widehat{G} \subset N_{\widehat{D}}(\widehat{\CO})$, we have the following theorem. 

\begin{thm}\label{thm:emb-numbers} Keeping the above notations, let $G := G_1$ and $h_{\gO}$ be the class number of $\gO$. Then we have
\begin{align*}
\sum_{i = 1}^{h} m(\gO, \CO_i; G_i) = \frac{2h_\gO}{[H : H \cap \N_{E/F}(E^\times)\CO_F^{\times+}]}\prod_{v<\infty}m(\gO_v, \CO_v; \CO_v^\times),
\end{align*}
where $H := \Nrd(G)$, and $m(\gO_v, \CO_v; \CO_v^\times)$ is the local embedding number at the place $v$. (Here $v$ runs over all finite places.)
\end{thm}

\begin{proof} By applying Lemma~\ref{lem:mO1} with $G = \CO_{+}^\times$, we have
\begin{align*}
m(\gO, \CO; \CO_{+}^\times) &= m(\gO, \CO; \CO^\times) [\Nrd(\CO^\times) : \Nrd(\CO_{+}^\times)\N_{E/F}(\gO^\times)]\\
&= m(\gO, \CO; \CO^\times) [\Nrd(\CO^\times) : \CO_F^{\times +}] = 2 m(\gO, \CO, \CO^\times).
\end{align*}
The latter equality follows from the fact that $D$ is ramified at all but one archimedian place, the Norm Theorem~\cite[Chap. III, Th\'eor\`eme 4.1]{vig80}
and the Dirichlet unit theorem. 

Now we return to the situation  $\CO_{+}^\times \subset G \subset N_D(\CO)_{+}$. Combining the above identity with Lemma~\ref{lem:mO2}, 
we have $$2 m(\gO, \CO; \CO^\times) = m(\gO, \CO; G)[\Nrd(G) : \Nrd(G) \cap \N_{E/F}(E^\times)\CO_F^{\times + }].$$
A similar identity holds for the other maximal orders. In other words, for each maximal order $\CO_i$, we have
$$2 m(\gO, \CO_i; \CO_i^\times) = m(\gO, \CO; G_i)[\Nrd(G_i) : \Nrd(G_i) \cap \N_{E/F}(E^\times)\CO_F^{\times + }].$$
However, the group $\Nrd(G_i)$ is independent of $i$ again by the Norm Theorem. Hence setting $H := \Nrd(G)$, we get
$$2m(\gO, \CO_i; \CO_i^\times) = m(\gO, \CO_i; G_i)[H : H \cap \N_{E/F}(E^\times)\CO_F^{\times + }].$$
So, we have
\begin{align*}
\sum_{i = 1}^h m(\gO, \CO_i; G_i) = \frac{2}{[H : H \cap \N_{E/F}(E^\times)\CO_F^{\times +}]}\sum_{i=1}^h m(\gO, \CO_i; \CO_i^\times). 
\end{align*}
We then apply~\cite[Chap. III, Th\'eor\`eme 5.11]{vig80} or \cite[Theorem 30.7.3]{voi18} to conclude the proof.

\end{proof}

\subsection{Torsion in maximal arithmetic groups}\label{subsec:torsion-elts}

 From now on, we will assume that the field $F$ has narrow class number one.
However, the results discussed here can be easily adapted to any field by following~\cite[\S\S 31 and 39]{voi18} given that our 
maximal orders do {\it not} satisfy the selectivity condition in~\cite{cf99}.

Since $F$ has narrow class number one, under the assumptions of Theorem~\ref{thm:emb-numbers}, we have
$$m(\gO, \CO; G) = \frac{2h_\gO}{[H : H \cap \N_{E/F}(E^\times)\CO_F^{\times+}]}\prod_{v<\infty}m(\gO_v, \CO_v; \CO_v^\times).$$

\begin{thm}\label{thm:elliptic-elts} Let $q \ge 2$ be an integer, and $e_q$ the number of elliptic points of order $q$ in $G$. 
Suppose that $e_q > 0$. For $q \ge 3$, let $E = F(\zeta_q)$, where $\zeta_q$ is a primitive $q$-th root of unity, and
let $\mathscr{S}_q$ be the set of $\CO_F$-orders defined by 
$$\mathscr{S}_q := \left\{ \CO_F[\zeta_q] \subset \gO \subset \CO_E : \#\gO_{\rm tors}^\times = q\right\}.$$
For $q = 2$, let $\mathscr{N}_q$ be a set of representatives for the norms of elements in $G$ in $\Nrd(W_{+})$,
and let $\mathscr{S}_q$ be the set of $\CO_F$-orders defined by 
$$\mathscr{S}_q := \bigcup_{n \in \mathscr{N}_q \atop E = F(\sqrt{-n})}\left\{ \CO_F[\sqrt{-n}] \subset \gO \subset \CO_E\right\}.$$
Then the number of elliptic points of order $q$ in $G$ is given by  
$$e_q := \frac{1}{2}\sum_{\gO \in \mathscr{S}_q} m(\gO, \CO; G).$$
\end{thm}

\begin{proof}
The proof is essentially an adaptation of the discussion of~\cite[\S 39.4]{voi18} (see also~\cite[Chap. IV, Section 2]{vig80}); the only 
difference arises from the elliptic points that are fixed by the Atkin-Lehner group $W_{+}$. However, the number of $2$-torsion 
elliptic elements can be computed by combining Theorem~\ref{thm:chinburg-friedman} and \S \ref{subsec:atkin-lehner-gp}.
\end{proof}

\begin{rem}\rm There seems to be very little discussion on the number of elliptic elements (or optimal embeddings) 
in maximal arithmetic Fuchsian groups. The only literature we could find on this topic is from Michon~\cite{mich81}
and Vign\'eras~\cite[Chap. IV, \S 3]{vig80} for $F = \Q$, and Maclachlan~\cite{mac06, mac09} for $[F:\Q] > 1$. In the
latter case, however, the presentation is very different than ours. Our results are stated in a way as to draw the most parallel 
with optimal embeddings in Fuchsian groups, which correspond to Shimura curves, given that there is an abundance of 
literature in this case (see~\cite[\S 30]{voi18} and references therein). 
\end{rem}

\subsection{Genus formula}\label{subsec:genus-formula}
Let $\Gamma$ be a Fuchsian group of signature $(g; e_1,\ldots, e_r)$, then the quotient $\Gamma\backslash \mathfrak{H}$ 
is a compact Riemann surface, whose volume is given by
$$\mathrm{Vol}(\Gamma\backslash \mathfrak{H}) = 2 \pi \left(2g - 2  + \sum_{i = 1}^r \left(1 - \frac{1}{e_i}\right) \right).$$
When $\Gamma = \Gamma_\CO$ is maximal in some commensurability class $\mathscr{C}(D)$, the volume 
depends only on $F$ and $S_f$ according to Theorem~\ref{thm:borel-volume}. In fact, it follows from 
Maclachlan~\cite[Corollary 5.7]{mac09} that all maximal arithmetic Fuchsian groups in the commensurability class 
$\mathscr{C}(D)$ have the same signature, and we can compute their genus by combining the volume formula in 
Theorem~\ref{thm:borel-volume} with the results  of Subsection~\ref{subsec:torsion-elts} (at least when $F$ has 
narrow class number one).

\section{\bf Shimura curves}\label{sec:shimura-curves}
We keep the notations of Section~\ref{sec:fuchsian-groups}. Here, we summarise the necessary backgrounds on canonical models 
and $\gp$-adic unformisation of Shimura curves. Our main references are~\cite{bc91, bz, car86b, nek12, sij13}. We view $F$
as a subfield of $\C$ via the embedding $v_1:\, F \hookrightarrow \C$.

\subsection{Complex uniformisation}\label{subsec:C-uniformisation} Let $U = \prod_{\gq}U_\gq \subset \widehat{\CO}^\times$
be a compact open subgroup, such that $U_\gp$ is maximal. We consider the quotient
$$X_U(\C) := D^\times\backslash X\times \widehat{D}^\times/U,$$
where $X := \mathbf{P}^1(\C) - \mathbf{P}^1(\R) = \mathfrak{H}^+ \sqcup \mathfrak{H}^-$, and $\mathfrak{H}^-$ and $\mathfrak{H}^+$
are the lower and upper Poincar\'e half-planes. Since $D$ is a division algebra, $X_U(\C)$ is a Riemann surface. 

There is a right action of $\widehat{D}^\times$ on $X \times \widehat{D}^\times$ by conjugation. For each $g \in \widehat{D}^\times$,
this induces an isomorphism of complex curves
$$X_U(\C) \stackrel{\sim}{\to} X_{g^{-1}Ug}(\C).$$

By the strong approximation theorem, we have the following bijections
$$D_{+}^\times\backslash \widehat{D}^\times/U \simeq D^\times\backslash \{\pm 1\}\times\widehat{D}^\times/U \simeq 
F_{+}^\times\backslash \widehat{F}^\times/\Nrd(U).$$
By class field theory, there is a unique abelian extension $F_U$ of $F$ such that the Artin map induces an isomorphism
$$\Art_F:\,\Gal(F_U/F) \simeq F_{+}^\times\backslash \widehat{F}^\times/\Nrd(U).$$
So the set $F_{+}^\times\backslash \widehat{F}^\times/\Nrd(U)$ is a Galois set. Therefore, there is a finite \'etale scheme
$\mathscr{T}_U$ defined over $F$ such that 
$$\mathscr{T}_U(F_U) = \mathscr{T}_U(\overline{F}) = \mathscr{T}_U(\C) = F_{+}^\times\backslash \widehat{F}^\times/\Nrd(U).$$

Shimura~\cite{shi70} shows that $X_U(\C)$ admits a canonical model defined over $F$ (see also~\cite{del71}). Namely, we have the following result.

\begin{thm}\label{thm:canonical-models} There is a curve $X_U$ defined over $F$, called a {\rm canonical model}, which satisfies the following properties:
\begin{enumerate}[(i)]
\item The set of complex points of $X_U$ is $X_U(\C)$, i.e. $$(X_U \otimes_{F,v_1}\C)(\C) = X_U(\C).$$
\item For a compact open $U'\subset U$, the morphism $X_{U'}(\C) \to X_U(\C)$ is induced by an $F$-morphism $X_{U'} \to X_U$.
\item For each $g \in \widehat{D}^\times$, the morphism $X_U(\C) \to X_{g^{-1}Ug}(\C)$ is induced from a $F$-morphism $X_U \to X_{g^{-1}Ug}$.
\item The morphism $X_U(\C) \to \mathscr{T}_U(\C)$, has connected fibres, and is induced by a morphism of $F$-schemes
$X_U \to \mathscr{T}_U$. In particular, the group of connected component $\pi_0(X_U)$ is a finite \'etale group scheme over $F$ 
such that $\pi_0(X_U)(\C) = \pi_0(X_U(\C))= \mathscr{T}_U(\C)$, where $\pi_0(X_U(\C))$ is the group of connected components of $X_U(\C)$.
\end{enumerate}
\end{thm}

\begin{proof}This is essentially a summary of the properties of canonical models of Shimura curves listed in~\cite[\S\S 1.1 and 1.2]{car86b}. 
\end{proof}

Theorem~\ref{thm:canonical-models} (iv) is known as the Shimura reciprocity law, it implies that $X_U$ is an irreducible scheme, which is not geometrically 
irreducible in general. However, when $\Nrd(U) = \widehat{\CO}_F^\times$, then $X_U$ is geometrically irreducible since we assume that $F$ has narrow class
number one.

We define the {\it adelic Atkin-Lehner group} by $\widehat{W} := N_{\widehat{D}}(U)/\widehat{F}^\times U$. By making use of the weak approximation theorem,
one can show that $$\widehat{W} \simeq \prod_{\gq \in S_f \cup S_0}\Z/2\Z,$$
where $S_0$ is the set of primes where $U_\gq$ is non-maximal.

\begin{cor}\label{cor:global-atkin-lehner} The group $\widehat{W}$ acts on $X_U(\C)$. This action is induced from an action of $\widehat{W}$ on $X_U$
defined over $F$.  In particular, if $W'\subseteq \widehat{W}$ is a subgroup, then the quotient $X_U/W'$ is defined over $F$.
\end{cor}

\begin{proof}Every element $g\in \widehat{W}$ defines an automorphism of $X_U(\C)$. By Theorem~\ref{thm:canonical-models} (iii),
this automorphism descends to $F$. 
\end{proof}

When there is an integral ideal $\gN$ coprime with the discriminant $\mathrm{disc}(D)$ of $\CO$, and an Eichler order $\CO_0(\gN) \subset \CO$
of level $\gN$ such that $U = \widehat{\CO_0(\gN)}^\times$, we will denote the Shimura curve $X_U$ by $X_0^D(\gN)$, or simply write $X_0^D(1)$
when $\gN = (1)$.

\subsection{Bruhat-Tits tree} Let $\mathcal{T}_\gp$ be the Bruhat-Tits tree attached to $\GL_2(F_\gp)$. 
Its set of vertices $\mathcal{V}(\mathcal{T}_\gp)$ consists of maximal $\CO_{F_\gp}$-orders in $\mathrm{M}_2(F_\gp)$, two vertices being adjacent if 
their intersection is an Eichler order of level $\gp$. Let $\vv{\mathcal{E}}(\mathcal{T}_\gp)$ denote the set of ordered edges of $\mathcal{T}_\gp$, 
i.e., the set of ordered pairs $(s, t)$ of adjacent vertices of $\mathcal{T}_\gp$. If $e = (s, t)$, the vertex $s$ is called the {\it source} of $e$ and the vertex $t$ 
is called its {\it target}; they are denoted by $s(e)$ and $t(e)$ respectively.

The Atkin-Lehner involution $\iota:\,\vv{\mathcal{E}}(\mathcal{T}_\gp) \to \vv{\mathcal{E}}(\mathcal{T}_\gp)$ sends the edge $e = (s, t)$ to the opposite
edge $\bar{e}$. We let $\mathcal{E}(\mathcal{T}_\gp) = \vv{\mathcal{E}}(\mathcal{T}_\gp)/\langle \iota \rangle$ be the set of non-oriented edges. 

The tree $\mathcal{T}_\gp$ is endowed with a natural left action of $\PGL_2(F_\gp)$ by isometries corresponding to conjugation
of maximal orders by elements of $\GL_2(F_\gp)$. This action is transitive on both  $\mathcal{V}(\mathcal{T}_\gp)$ and 
$\vv{\mathcal{E}}(\mathcal{T}_\gp)$.

\subsection{$\gp$-adic uniformisation}\label{subsec:l-adic-uniformisation}

Let $\overline{F_\gp}$ be an algebraic closure of $F_\gp$, and $\C_\gp := \widehat{\overline{F_\gp}}$ be a fixed completion of 
$\overline{F_\gp}$. Let $\widehat{\mathscr{H}}_\gp$ be $\gp$-adic upper half plane. This is the formal scheme over $\Spf(\CO_{F_\gp})$
defined in~\cite[\S 1.3]{bc91} by
$$\widehat{\mathscr{H}}_\gp := \mathbf{P}^1(\C_\gp) - \mathbf{P}^1(F_\gp).$$
The scheme $\widehat{\mathscr{H}}_\gp$ admits a natural action by the group $\GL_2(F_\gp)$, which factors through the adjoint group 
$\PGL_2(F_\gp)$. We let 
$$\widehat{\mathscr{H}}_\gp^{\rm ur} = \widehat{\mathscr{H}}_\gp\times_{\Spf(\CO_{F_\gp})}\Spf(\CO_{F_\gp}^{\rm ur}).$$

Let $B$ be the totally definite quaternion algebra defined over $F$ whose set of ramified {\it finite} places is $S_f \setminus \{\gp\}$ 
so that $B_\gp \simeq \mathrm{M}_2(F_\gp)$. (Note that this means that the set of ramified archimedian places of $B$ is $S_\infty \cup \{v_1\}$.)
We write $\widehat{B} = B_\gp \times B^{\gp}$ and $\widehat{D} = D_\gp \times D^{\gp}$, and we fix an isomorphism 
$\varphi:\, D^{\gp} \stackrel{\sim}{\to} B^{\gp}$. We let $K = K_\gp \times K^{\gp}$ be a compact open subgroup of $\widehat{B}^\times$ such that 
$K_\gp \simeq \GL_2(\CO_{F_\gp})$ and $\varphi(U^{\gp}) = K^{\gp}$. We also let 
$$K_\gp^0 := \left\{\begin{pmatrix} a& b \\ c & d \end{pmatrix} \in K_\gp : c \equiv 0 \bmod \gp \right\},$$
and $K_0(\gp) = K_\gp^0 \times K^\gp$. 

Since $U_\gp$ is the maximal compact open subgroup of $D_\gp^\times$, the norm map induces an isomorphism 
$D_\gp^\times/U_\gp \stackrel{\sim}{\to} F_\gp^\times/\CO_{F_\gp}^\times$ (see \cite[Chap. II, Lemme 1.5]{vig80}). 
The group $B_\gp^\times$ acts on $F_\gp^\times/\CO_{F_\gp}^\times$ through its reduced norm map $\Nrd:\, B_\gp^\times \to F_\gp^\times$. 
We obtain a corresponding action of $B_\gp^\times$ on $D_\gp^\times/U_\gp$. This, together with the isomorphism $\varphi$, gives an action
of $\widehat{B}^\times$ on $\widehat{D}^\times/U$. 

\begin{thm}[\v{C}erednik-Drinfel'd]\label{thm:cerednik-drinfeld}
There exist a model $\mathscr{M}$ of $X_U$ over $\CO_{F_\gp}$, and an isomorphism of formal schemes
$$\widehat{\mathscr{M}}^{\rm ur} = \widehat{\mathscr{M}}\times_{\Spf(\CO_{F_\gp})}\Spf(\CO_{F_\gp}^{\rm ur})
\simeq B^\times\backslash \widehat{\mathscr{H}}_\gp^{\rm ur} \times \widehat{D}^\times/U,$$
where $\widehat{\mathscr{M}}$ is the completion of $\mathscr{M}$ along its special fibre.
\end{thm}

\begin{proof} See~\cite[Theorem 3.1]{bz}.
\end{proof}

\subsection{The dual graph} The {\it dual graph} associated to $B^\times\backslash \widehat{\mathscr{H}}_\gp^{\rm ur} \times \widehat{D}^\times/U$ 
is the weighted graph $$\mathcal{G} := B^\times\backslash \mathcal{T}_\gp \times \widehat{D}^\times/U.$$
The vertices of $\mathcal{V}(\mathcal{G})$ and oriented edges $\vv{\mathcal{E}}(\mathcal{G})$ of $\mathcal{G}$ are given 
respectively by
$$\mathcal{V}(\mathcal{G}) := B^\times\backslash \mathcal{V}(\mathcal{T}_\gp) \times \widehat{D}^\times/U,\,\,\text{and}\,\,
\vv{\mathcal{E}}(\mathcal{G}) := B^\times\backslash \vv{\mathcal{E}}(\mathcal{T}_\gp) \times \widehat{D}^\times/U.$$
We define the weight of a vertex $v \in \mathcal{V}(\mathcal{G})$ to be $\#\Stab_{B^\times/F^\times}(v)$, and the weight of
an edge $e \in \vv{\mathcal{E}}(\mathcal{G})$ to be $\#\Stab_{B^\times/F^\times}(e)$. For a vertex $v$, we let $\Star(v)$
denote the set of all edges containing $v$.

\begin{prop}\label{prop:bi-partite-graph} The maps
\begin{align*}
\vartheta_1:\, (B_\gp^\times/F_\gp^\times K_\gp) \times (D_\gp^\times/U_\gp) \times ({D^{\gp}}^\times\!/U^{\gp}) &\to (B_\gp^\times/K_\gp) \times \Z \times ({B^\gp}^\times\!/K^\gp)\\
(x_\gp, y_\gp, y^\gp) &\mapsto (x_\gp, \ord_\gp(\Nrd(y_\gp)), \varphi(y^\gp)),\\
\vartheta_2:\, (B_\gp^\times/F_\gp^\times K_\gp^0) \times (D_\gp^\times/U_\gp) \times ({D^{\gp}}^\times\!/U^{\gp}) &\to (B_\gp^\times/K_\gp^0) \times \Z \times ({B^\gp}^\times\!/K^\gp)\\
(x_\gp, y_\gp, y^\gp) &\mapsto (x_\gp, \ord_\gp(\Nrd(y_\gp)), \varphi(y^\gp))
\end{align*}
induce an isomorphism of bipartite graphs
\begin{align*}
\mathcal{V}(\mathcal{G}) &= B^\times\backslash \mathcal{V}(\mathcal{T}_\gp) \times \widehat{D}^\times/U \stackrel{\sim}{\to}
(B^\times \backslash \widehat{B}^\times/K) \times \Z/2\Z,\\
\vv{\mathcal{E}}(\mathcal{G})&= B^\times\backslash \vv{\mathcal{E}}(\mathcal{T}_\gp) \times \widehat{D}^\times/U \stackrel{\sim}{\to}
(B^\times \backslash \widehat{B}^\times/K_0(\gp)) \times \Z/2\Z\,
\end{align*} 
as follows: We write 
$\mathcal{V}(\mathcal{G}) = \mathcal{V} \sqcup \mathcal{V}' \simeq B^\times \backslash \widehat{B}^\times/K \sqcup 
B^\times \backslash \widehat{B}^\times/K$, and we let the adjacency matrix in the basis $\mathcal{V} \cup \mathcal{V}'$ be
given by the matrix
$$\begin{bmatrix} 0 & T_\gp \\ T_\gp &0 \end{bmatrix},$$
where $T_\gp$ is the Hecke operator at $\gp$ acting on the Brandt module $M := \Z[B^\times \backslash \widehat{B}^\times/K]$. 
In that identification, the action of the Atkin-Lehner involution $w_\gp$ on $\mathcal{V}(\mathcal{G})$ is given by the matrix
$$\begin{bmatrix} 0 & \mathbf{1}_M \\ \mathbf{1}_M &0 \end{bmatrix}.$$
\end{prop}

\begin{proof}
See~\cite[Propositions 3.1.8 and 3.1.9]{sij13}, \cite[\S 1.5]{nek12} or \cite[\S 4]{kur79}.
\end{proof}

\begin{rem}\rm In the isomorphism of Proposition~\ref{prop:bi-partite-graph}, the set of non-oriented edges is given by
$$\mathcal{E}(\mathcal{G}) = B^\times\backslash \mathcal{E}(\mathcal{T}_\gp) \times \widehat{D}^\times/U \simeq
B^\times \backslash \widehat{B}^\times/K_0(\gp).$$
\end{rem}

The following result is an essential ingredient in the description of the special fibre of the \v{C}erednik-Drinfel'd model described in
Theorem~\ref{thm:cerednik-drinfeld}. As we will see later, it is also useful in understanding the automorphism group of the curve $X_U$.

\begin{thm}\label{thm:special-fibre} 
Let $\mathscr{M}$ be the scheme in Theorem~\ref{thm:cerednik-drinfeld}. Then, we have the following:
 \begin{enumerate}[(i)]
\item $\mathscr{M}\otimes_{\CO_{F_\gp}}\!\CO_{F_{\gp^2}}$ is a normal, proper, flat and semistable scheme over $\CO_{F_{\gp^2}}$.
\item The special fibre of $\mathscr{M}\otimes_{\CO_{F_\gp}}\!\CO_{F_{\gp^2}}$ is reduced. Its components are rational curves, and all
its singular points are ordinary double points. 
\item The weighted dual graph associated to $\mathscr{M}\otimes_{\CO_{F_\gp}}\!\CO_{F_{\gp^2}}$ is the graph $\mathcal{G}$ 
described in Proposition~\ref{prop:bi-partite-graph}.
\item Let $\mathcal{H}$ be a connected component of $\mathcal{G}$, and $\mathscr{M}_{\mathcal{H}}$ the corresponding irreducible
component of $\mathscr{M}$. Then the arithmetic genus of $\mathscr{M}_{\mathcal{H}}$ is given by the Betti number 
$1 + \#\mathcal{E}(\mathcal{H})- \#\mathcal{V}(\mathcal{H})$. 
\end{enumerate}
\end{thm}

\begin{proof}
See~\cite[Proposition 1.5.5]{nek12} or~\cite[Proposition 3.2]{kur79}.
\end{proof}

\subsection{Special fibre of $\mathscr{M}\otimes_{\CO_{F_\gp}}\!\CO_{F_{\gp^2}}$} The curve $\mathscr{M}$ is an {\it admissible} curve over $\CO_{F_\gp}$
in the following sense:
\begin{enumerate}[(i)]
\item $\mathscr{M}\otimes_{\CO_{F_\gp}}\!\CO_{F_{\gp^2}}$ is a normal, proper, flat and semistable scheme over $\CO_{F_{\gp^2}}$. Each irreducible
component has a smooth generic fibre. 
\item The completion of the local ring of $\mathscr{M}\otimes_{\CO_{F_\gp}}\!\CO_{F_{\gp^2}}$ at each of its singular points $x$ is isomorphic, as an 
$\CO_{F_\gp}$-algebra, to $\CO_{F_\gp}[[X,Y]]/(XY - \varpi_\gp^w)$, where $\varpi_\gp$ is a uniformising element at $\gp$, and $w = w(x) \in \{1, 2, 3,\,\ldots\}$.
\item The special fibre $\mathscr{M}\otimes_{\CO_{F_\gp}} k(\gp)$ is reduced; the normalisation of each of its irreducible components is isomorphic to
$\mathbf{P}^1(k(\gp))$; its only singular points are ordinary double points, where $k(\gp)$ is the residue field of $\CO_{F_{\gp^2}}$.
\end{enumerate}
The dual graph encodes the following combinatoric data of the special fibre.

\begin{enumerate}
\item[(iv)] Each vertex $v \in \mathcal{V}(\mathcal{G})$ corresponds to an irreducible component $C_v$ of the special fibre $\mathscr{M}\otimes_{\CO_{F_\gp}} k(\gp)$. 
\item[(v)] Each edge $e = \{v, v'\} \in \mathcal{E}(\mathcal{G})$ corresponds to a singular point in $x_e \in C_v \cap C_{v'}$. The completion of local ring at $x_e$ is of the form 
$\CO_{F_\gp}[[X,Y]]/(XY - \varpi_\gp^w)$, where $w = w(e)$ is the weight of the edge $e$.
\end{enumerate}
The above description can be found in~\cite{nek12, kur79} and references therein.

\subsection{Automorphism groups}\label{subsec:automorphism-group}
An {\it automorphism of weighted graph} $\mathcal{G}$ is an automorphism of graphs
which preserves the weights of the edges. We will denote the group of such automorphisms by $\Aut(\mathcal{G})$. We
note that there is a natural inclusion $\Aut(\mathcal{G}) \subset \Aut^s(\mathcal{G})$, where $\Aut^s(\mathcal{G})$ is the 
automorphism group of the underlying simple graph to $\mathcal{G}$.

For the next statement, we recall the notion of admissibility from~\cite{kr08}. We say that an element $\omega \in \Aut(\mathcal{G})$
is {\it admissible} if there is no vertex $v \in \mathcal{V}(\mathcal{G})$ fixed by $\omega$ such that $\Star(v)$ has at least $3$
edges also fixed by $\omega$. We say that a subgroup $H \subset \Aut(\mathcal{G})$ is {\it admissible} if every non-trivial element 
$\omega \in H$ is admissible.

\begin{prop}\label{prop:quotient-graph}
Let $W' \subset \widehat{W}$ be a subgroup. Then, we have the following:
\begin{enumerate}[(1)]
\item The dual graph of $(\mathscr{M}/W')\otimes_{\CO_{F_\gp}}\!\CO_{F_{\gp^2}}$ is the graph $\mathcal{G}' = (\mathcal{G}/W')^*$,
where ${}^*$ means we remove all loops from the quotient graph $\mathcal{G}/W'$.
\item Assume that the genus of $X_U/W'$ is at least $2$, and let $\mathcal{G}_{st}$ be the dual graph of the stable model 
$(\mathscr{M}/W')_{st}$ of $\mathscr{M}/W'$. Then there is a natural injection $\varrho:\,\Aut(X_U/W') \hookrightarrow \Aut(\mathcal{G}_{st})$
whose image $\mathrm{im}(\varrho)$ lies in an admissible subgroup.
\end{enumerate}
\end{prop}

\begin{proof}
Part (1) follows from general properties of Mumford curves. From~\cite[Lemmas 1.12 and 1.16]{dm69}, and universal properties of stable
models, there is an injection $\varrho:\,\Aut(X_U/W') \hookrightarrow \Aut(\mathcal{G}_{st})$. To prove Part (2), we only need to show that
every non-trivial element in the image of $\varrho$ is admissible. To this end, let $\omega \in \Aut(X_U/W')$ be such that $\varrho(\omega)$ fixes
a vertex $v$, and at least 3 edges in $\Star(v)$. Then, since every automorphism of the projective line, which fixes at least $3$ points is the identity, 
the restriction $\omega|_{C_v}$ is the identity, where $C_v$ is the irreducible component associated to $v$. This would imply that, as an automorphism 
of the Riemann surface $(X_U/W')(\C)$, $\omega$ fixes more than $2g(X_U/W') + 2$ points, where $g(X_U/W')$ is the genus of $X_U/W'$. 
Hence $\omega$ must be the identity. Therefore, if $\omega$ is non-trivial, then $\varrho(\omega)$ must be admissible. 
\end{proof}

\section{\bf The hyperelliptic Shimura quotient curve}\label{sec:gross-curve}

\subsection{The quaternion algebra}\label{subsec:quat-algs}
Let $F = \Q(\alpha) = \Q(\zeta_{32}+\zeta_{32}^{-1})$ be the maximal totally real subfield of the cyclotomic field of the $32$nd roots of unity. 
This field is defined by the polynomial $x^8 - 8x^6 + 20x^4 - 16x^2 + 2$. Let $\sigma$ be a generator of $\Gal(F/\Q)$. Let $\CO_F$ be the ring 
of integers of $F$. Let  $v_1,\,\ldots,v_8$ be the real places of $F$. We consider the quaternion algebra $D/F$ ramified at $v_2,\ldots, v_8$ and 
the unique prime $\gp$ above $2$. More concretely, we have $D = \left(\frac{u, -1}{F}\right)$, where $u = -\alpha^2 + \alpha$ has signature 
$(+, -,\ldots,-)$. Let $\CO$ be the maximal order in $D$ given by
$$\CO := \CO_F[1, i, \frac{(\alpha^7 + \alpha^6 + \alpha^4 + 1) + \alpha^7i + j}{2}, \frac{(\alpha^7 + \alpha^6 + \alpha^4 + 1)i + k}{2}].$$ 

We also let $B/F$ be the totally definite quaternion algebra ramified exactly 
at all the real places $v_1,\ldots,v_8$, and fix a maximal order $\CO_B$ in $B$. Both these orders were computed using the
Quaternion Algebras Package in Magma~\cite{magma} implemented by Voight~\cite{voi05}).

\subsection{The CM field and its embedding}\label{subsec:CM-extension}
We recall the following diagram
\begin{eqnarray*}
\begin{tikzcd}
&&\Q(\zeta_{64})&&\\
\Q(\zeta_{64})^{+}\arrow[-, "2"]{urr}&&\Q(i(\zeta_{64} + \zeta_{64}^{-1}))\arrow[-, "2"']{u}&&\Q(\zeta_{32})\arrow[-, "2"']{ull}\\
&&F\arrow[-, "2"']{ull}\arrow[-, "2"']{u}\arrow[-, "2"]{urr}&&\\
\end{tikzcd}
\end{eqnarray*}
The subfield $K := \Q(\beta) = \Q(i(\zeta_{64} + \zeta_{64}^{-1}))$ is the unique CM extension of $F$ with class number $17$. 
For later, we observe that $\beta^2 = -2 - \alpha$, where $\gp = (2 + \alpha)$. Since $\gp$ is the unique prime of $F$ that
ramifies in both $K$ and $D$, we see that $K$ is a splitting field of $D$ by~\cite[Chap. III, Th\'eor\`eme 3.8]{vig80}. It is possible to
compute an explicit embedding $K \hookrightarrow D$ using the Quaternion Algebras Package in \verb|Magma|~\cite{magma} (see~\cite{voi05}),
but we will not need such a map here.

\subsection{The spaces of forms}\label{subsec:newforms}
Let $S_2(\gp)^{\rm new}$ be the new subspace of Hilbert cusp forms of level $\gp$ and weight $2$, this is a $40$-dimensional space. 
Let $S_2^D(1)$ (resp. $S_2^{B}(\gp)^{\rm new}$) be the space of automorphic forms of level $(1)$ and weight $2$ on $D$ (resp. new subspace
of automorphic forms of level $\gp$ and weight $2$ on $B$.) By the Jacquet-Langlands correspondence, we have isomorphisms of Hecke modules
$$S_2(\gp)^{\rm new} \simeq S_2^D(1) \simeq S_2^{B}(\gp)^{\rm new}.$$
The space $S_2(\gp)^{\rm new}$ decomposes into $5$ Hecke constituents of dimensions 
$4, 4, 4, 4$ and $24$ respectively. (We note that all the computations have been performed using the Hilbert Modular Forms Package in 
\verb|Magma|~\cite{magma}, the algorithms are described in~\cite{dd08, dv13, gv11}.) There are choices of newforms $f, f', g, g'$ and $h$
in those constituents such that we have:

\begin{table}
\caption{Newforms of level $\gp$ and weight $2$ on $F = \Q(\zeta_{32})^{+}$}
\begin{tabular}{@{}cccc@{}}\toprule
\text{Newform}&\text{Coefficient field}\,\, $L_f$&\text{Fixed field}\,\, $K_f = L_f^\Delta$ & $\Gal(L_f/K_f)$ \\\midrule
$f$, $f'$ & $\Q(\zeta_{15})^+$& $\Q$ & $\Z/4\Z$\\
$g$, $g'$ & \text{Quartic subfield of}\, $\Q(\zeta_{95})^+$& $\Q$ & $\Z/4\Z$\\
\multirow{2}{*}{$h$}  &\text{Ray class field of modulus} &  $\Q(c) := \Q[x]/(r(x))$,& \multirow{2}{*}{$\Z/8\Z$}\\
& $\gc = (\frac{1}{2}(c^2 - 16c + 25))$ & $ r = x^3 + x^2 - 229x + 167$& \\\midrule
\text{Relations}&${}^\sigma\!f = f'\,\text{and}\, {}^{\sigma^2}\!f = f^\tau$ & ${}^\sigma\!g = g'\, \text{and}\, {}^{\sigma^2}\!g = g^\tau$ & ${}^\sigma\!h = h^\tau$\\
\bottomrule
\end{tabular}
\label{table:eigenforms}
\end{table}

\begin{enumerate}[(i)]
\item The forms $f$ and $f'$ have the same coefficient field $L_f = L_{f'}$, which is the real quartic field $\Q(\zeta_{15})^+$ 
given by $x^4 + x^3 - 4x^2 - 4x + 1$. They satisfy the relations ${}^\sigma\!f = f'$ and ${}^{\sigma^2}\!f = f^\tau$,
where $\tau$ is a generator of $\Gal(L_f/\Q)$. 
\item  The forms $g$ and $g'$ have the same coefficient field $L_g = L_{g'}$, which is the real quartic subfield of $\Q(\zeta_{95})^+$
given by $x^4 + 19x^3 - 59x^2 + 19x + 1$. They satisfy the relations ${}^\sigma\!g = g'$ and ${}^{\sigma^2}\!g = g^\tau$,
where $\tau$ is a generator of $\Gal(L_g/\Q)$. 
\item The coefficient field of the form $h$ is a field $L_h$ of degree $24$, which is cyclic over the field $K_h = \Q(c)$ defined by 
$c^3 + c^2 - 229c + 167 = 0$. More precisely, it is the ray class field of conductor $\gc = (\frac{1}{2}(c^2 - 16c + 25))$. The form $h$
satisfies the relation ${}^\sigma\!h = h^\tau$, where $\tau$ is a generator of $\Gal(L_h/K_h)$. 
\end{enumerate} 
(We summarise that data in Table~\ref{table:eigenforms}, and the relations among the forms.)
Let $w$ and $w_D$ be the Atkin-Lehner involutions acting on $S_2(\gp)^{\rm new}$ and $S_2^D(1)$ respectively.
The Atkin-Lehner involution $w$ acts as follows:
\begin{align*}
w f =  -f,\,\,w f' = -f',\,\,w g  = -g,\,\,w g'  = -g',\,\,w h  = h.
\end{align*}
We recall that $w_D = -w$.

\subsection{The Shimura curve and its quotient}\label{subsec:curves}
Let $X_0^D(1)$ be the Shimura curve attached to $\CO$.  Let $w_D$ be the Atkin-Lehner involution at $\gp$, and $C := X_0^D(1)/\langle w_D \rangle$. 
We can canonically identify $S_2^D(1)$ with the the space of $1$-differential forms on $X_0^D(1)$. From the discussion in \S\ref{subsec:newforms}, 
it follows that $X_0^D(1)$ is a curve of genus $40$; and that $C$ is a curve of genus $16$. 
 
 \begin{thm}\label{thm:signatures} 
 The curves $X_0^D(1)$ and $C$ have the respective signatures $(40; 3^{18}, 16^1)$ and $(16; 2^{17}, 3^9, 32^1)$. 
 \end{thm}
 
\begin{proof} The complex points of the curve $X_0^D(1)$ are determined by the quotient $\Gamma^1\backslash \mathfrak{H}$,
where $\Gamma^1$ is the image of $\CO^1$ inside $\PSL_2(\R)$. So it is a Shimura curve. So, we can compute the signature
of $X_0^D(1)$ using the Shimura Curves Package in \verb|magma|~\cite{magma}, which was implemented by Voight~\cite{voi09}. 
This gives that $X_0^D(1)$ has signature $(40; 3^{18}, 16^1)$. 

The curve $C = X_0^D(1)/\langle w_D \rangle$ is given by the maximal arithmetic Fuchsian group $\Gamma_\CO$. It is not a Shimura
curve. Although Voight has implemented algorithms for computing with maximal arithmetic Fuchsian groups, they are not publicly available 
yet. So, we compute the signature of $C$ by using the results of Section~\ref{sec:fuchsian-groups}. 

Let $q>2$ be an integer. Then, by Theorem~\ref{thm:chinburg-friedman}, $\Gamma_\CO$ contains an elliptic element of order $q$ if and only
if the following three conditions are satisfied:
\begin{enumerate}[(i)]
\item $2\cos(2\pi/q) \in F$;
\item no prime $\gq \in S_f$ splits in $E = F(\zeta_q)$;
\item the ideal generated by $2 + 2\cos(2\pi/q)$ is supported at $S_f$ modulo squares.
\end{enumerate}
It is enough to test all integers $q$ between $3$ and $64$.  The only $q \ge 3$ which satisfy these three conditions are: $3, 4, 6, 8, 16$ and $32$.

For $q = 4, 8, 16$ or $32$, we have $E = F(\zeta_q) = \Q(\zeta_{32})$. In that case, the only $\CO_F$-order which contains $\CO_F[\zeta_{32}]$
and optimally embeds into $D$ is the maximal order $\CO_E$. By Theorem~\ref{thm:elliptic-elts}, we get that $e_{32} = 1$. 

For $q = 3$, we have $E = F(\frac{1 + \sqrt{-3}}{2})$.  In that case, the only $\CO_F$-order which contains $\CO_F[\frac{1+\sqrt{-3}}{2}]$
and optimally embeds into $D$ is also the maximal order $\gO := \CO_E$. We have $h_\gO = 9$. Now since the prime $\gp$ is inert
in the relative extension $E/F$, we have $[H : H \cap \N_{E/F}(E^\times)\CO_F^{\times +}] = 2$. So, by Theorem~\ref{thm:elliptic-elts}, 
we get that $e_{3} = 9$. 

Finally, for $q = 2$, we have $W^1 = W_{+} = \Z/2\Z$ since $F$ has narrow class number one and there is a unique prime in $S_f$; 
namely, the prime $\gp$ above $2$. So the unique CM extension $E/F$ which satisfies the condition of Theorem~\ref{thm:chinburg-friedman} 
is the extension $K$ discussed in Subsection~\ref{subsec:CM-extension}. Recall that the ideal $\gp$ is generated by the totally positive
element $n = 2 + \alpha$. The only $\CO_F$-order which contains $\CO_F[\sqrt{-n}]$ and optimally embeds into $D$ is also the maximal 
order $\gO := \CO_K$. We have $h_\gO = 17$. Now since the prime $\gp$ is ramified in the relative extension $K/F$, we have 
$[H : H \cap \N_{E/F}(E^\times)\CO_F^{\times +}] = 1$. So, by Theorem~\ref{thm:elliptic-elts}, we get that $e_{2} = 17$. 

So we conclude that there are 3 classes of elliptic elements in $\Gamma_\CO$ of orders $2, 3$ and $32$, with respective multiplicities 
$17$, $9$ and $1$. 

By the volume formula in Theorem~\ref{thm:borel-volume} and the genus formula in Subsection~\ref{subsec:genus-formula}, the genus $g$ of the
curve $C$ must satisfy the equality
$$\frac{\mathrm{Vol}(\Gamma_\CO\backslash \mathfrak{H})}{2\pi} = \frac{1455}{32} = 2(g - 1) + 17\left(1 - \frac{1}{2}\right) + 9\left(1 - \frac{1}{3}\right) + \left(1 - \frac{1}{32}\right).$$
Solving this, we get that $g = 16$. Hence the curve $C$ has signature $(16; 2^{17}, 3^9, 32^1)$. 
\end{proof}

\begin{lem}\label{lem:field-of-defn} 
The curve $X_0^D(1)$ and the Atkin-Lehner involution $w_D$ are both defined over $\Q$. In particular, the curve $C$ descends to $\Q$.
\end{lem}

\begin{proof}
Since $\sigma(\gp) = \gp$ and the ray class group of modulus $\gp v_2\cdots v_8$ is trivial, the curve $X_0^D(1)$ is defined over $F$ by~\cite[Corollary]{doi-naga67},
and the field of moduli is $\Q$. The field $\Q(\zeta_{32})$ is a splitting field for $D$ whose class number is one. So, there is a unique CM point attached to 
the extension $\Q(\zeta_{32})/F$, and it is defined over $F$. Therefore, by~\cite[Corollary 1.9]{sv16}, the curve $X_0^D(1)$ descends to $\Q$. 

\medskip
Alternatively, by using the moduli interpretation in~\cite{car86b}, or the more recent work~\cite{tixi16}, 
one can show that both $X_0^D(1)$ and $w_D$ are defined over $\Q$. 
\end{proof}


\subsection{The dual graph of the quotient curve} The dual graph $\mathcal{G}'$ of the curve $\mathscr{M}/\langle w_\gp \rangle$
is displayed in Figure~\ref{figure:dual-graph}. It was computed by using Proposition~\ref{prop:bi-partite-graph}
and Proposition~\ref{prop:quotient-graph}. The computations combine both \verb|Magma|~\cite{magma} and \verb|Sage|~\cite{sagemath}.

\begin{lem}\label{lem:dual-graph} 
The automorphism group of $\mathcal{G}'$ is $\Aut(\mathcal{G}') \simeq \Z/4\Z \ltimes (\Z/2\Z)^4$.
\end{lem}

\begin{proof} We compute the dual graph $\mathcal{G}'$ of the quotient $\mathscr{M}/\langle w_\gp \rangle$ using Proposition~\ref{prop:bi-partite-graph}
and Proposition~\ref{prop:quotient-graph}. Let $B$ be the definite quaternion algebra defined in \S \ref{subsec:quat-algs}. Then, by 
Proposition~\ref{prop:bi-partite-graph} and Proposition~\ref{prop:quotient-graph},  the dual graphs $\mathcal{G}$
and $\mathcal{G}'$ are determined by the Brandt module $M_B := \Z[B^\times \backslash \widehat{B}^\times/\widehat{\CO}_B^\times]$. 
In this case, the class number of the maximal order $\CO_B$ is $58$, and a basis of this module is given by equivalence classes of $\CO_B$-right ideals. 
We let $v_1, v_2, \ldots, v_{58}$ be such a basis, which we order so that the weights of the elements are in decreasing order. We get the following 
sequence of weights: $32, 24, 16, 8^2, 4^3, 3^4, 2^6$ and $1^{40}$, where the exponent indicates the number of times each weight is repeated. Similarly, 
we compute the set of edges, and we obtain the following sequence for their weights: $32, 16^2, 8^6, 4^6, 2^{12}$ and $1^{128}$. By combining 
this with the Hecke operator $T_\gp$, we obtain the graph $\mathcal{G}'$ in Figure~\ref{figure:dual-graph}. 

We compute the automorphism group $\Aut^s(\mathcal{G}')$ of the underlying simple graph using \verb|Magma|, and check that every element in 
$\Aut^s(\mathcal{G}')$ preserves the weights of the edges, i.e. that $\Aut(\mathcal{G}') = \Aut^s(\mathcal{G}')$. 

To determine the group structure of $\Aut(\mathcal{G}')$, we first check that there is a unique normal subgroup of $\Aut(\mathcal{G}')$ which is isomorphic to 
$(\Z/2\Z)^4$. Finally, we show that there is a unique cyclic subgroup of order $4$ whose intersection with $(\Z/2\Z)^4$ is the neutral element.
\end{proof}

\begin{rem} \rm As a byproduct of the computation of $\mathcal{G}'$, we check that 
$$1 + \#\mathcal{E}(\mathcal{G}') - \#\mathcal{V}(\mathcal{G}') = 1 + 73 - 58 = 16,$$ 
which is the genus of $\mathscr{M}/\langle w_\gp \rangle$, or equivalently $C$.
\end{rem}

Let $(\mathscr{M}/\langle w_\gp \rangle)_{st}$ be the stable model of $\mathscr{M}/\langle w_\gp \rangle$, and 
$\mathcal{G}_{st}$ its dual graph. By definition, $(\mathscr{M}/\langle w_\gp \rangle)_{st}$ is stable if for all $v \in \mathcal{V}(\mathcal{G}_{st})$,
we have $\#\Star(v) \ge 3$. So, we obtain $(\mathscr{M}/\langle w_\gp \rangle)_{st}$ by blowing down all components $C_v$ 
associated to a vertex $v$ such that $\#\Star(v) < 3$. On graphs, this corresponds to doing the following:
\begin{enumerate}[a)]
\item For all $v \in \mathcal{V}(\mathcal{G}')$, with $\#\Star(v) = 1$, remove $v$ and all edges in $\Star(v)$;
\item For each $v \in \mathcal{V}(\mathcal{G}')$, with $\#\Star(v) = 2$, contract the chain $v' \stackrel{e'}{-} v \stackrel{e''}{-} v''$ to
$v'\stackrel{e}{-} v''$ with $w(e) = w(e')+w(e'')$.
\end{enumerate}
By applying this process to the curve $\mathscr{M}/\langle w_\gp \rangle$, and then relabelling the resulting graph, we obtain the stable 
model whose dual graph $\mathcal{G}_{st}$ is given by Figure~\ref{figure:st-dual-graph}. The graph $\mathcal{G}_{st}$ has $30$ vertices
and $45$ edges so that
$$1 + \#\mathcal{E}(\mathcal{G}_{st}) - \#\mathcal{V}(\mathcal{G}_{st}) = 1 + 45 - 30 = 16.$$ 

\begin{lem}\label{lem:st-dual-graph} 
Let $(\mathscr{M}/\langle w_\gp \rangle)_{st}$ be the stable model of $\mathscr{M}/\langle w_\gp \rangle$, and $\mathcal{G}_{st}$ its dual graph.
Then, $\mathcal{G}_{st}$ is a connected graph such that  $\Aut(\mathcal{G}_{st}) = \Aut(\mathcal{G}')$.
\end{lem}

\begin{proof}
This follows from a direct calculation.
\end{proof}

\begin{lem}\label{lem:admissible}
Every admissible subgroup of $\Aut(\mathcal{G}_{st})$ of exponent $2$ has order $2$.
\end{lem}

\begin{proof} 
First, we note that, since the degree of the Hecke operator $T_\gp$ is $3$, and $(\mathscr{M}/\langle w_\gp \rangle)_{st}$ is stable,
$\#\Star(v) = 3$ for each $v \in \mathcal{V}(\mathcal{G}_{st})$.

In the notations of Figure~\ref{figure:st-dual-graph}, we label the vertices $1,2,\ldots, 30$. There are $19$ permutations of order 
$2$ in $\Aut(\mathcal{G}_{st}) \subset S_{30}$. Of those $19$ 
permutations, there are exactly $4$ with the same support of length $28$. Each of the remaining $17$ has a support whose
length belongs to $\{2, 4, 6, 8\}$. The permutations of length $28$ fix the vertices $v = 1$ and $v'=2$. So, they must be admissible
since a non-admissible element must fix at least $4$ different vertices. For each of remaining $17$ permutations, one easily sees that 
the complement of its support contains a vertex $v$ and its $\Star(v)$, meaning that it cannot be admissible.  

To conclude the proof of the lemma, we let $\sigma_i$, $i = 1, 2, 3, 4$ be the $4$ admissible permutations obtained above,
and we check that $\sigma_i \sigma_j$ is not admissible for $i \neq j$. 
\end{proof}

\begin{lem}\label{lem:AutC-embeds}
There is an injection $\Aut(C) \hookrightarrow H$ into an admissible subgroup of $\Aut(\mathcal{G}_{st})$ of exponent $2$.
In particular $\Aut(C)$ has order at most $2$. 
\end{lem}

\begin{proof} In Subsection~\ref{subsec:decomposition}, we will show that the endomorphism ring of each of the simple factor
of $\Jac(C)$ is a totally real field. (This follows from the decomposition~\eqref{eq:jac-decomp2}.) Using this, we see that 
$\Aut(C) \subset (\Z/2\Z)^4$. So, by Proposition~\ref{prop:quotient-graph}, $\Aut(C)$ injects into an admissible subgroup $H$ of 
$\Aut(\mathcal{G}_{st})$ of exponent $2$. By Lemma~\ref{lem:admissible}, $H$ has order at most $2$.
\end{proof}

\begin{center}
\begin{figure}
\caption{The dual graph $\mathcal{G}'$ of the quotient curve $\mathscr{M}/\langle w_\gp\rangle$.}
\includegraphics[width = .72\textwidth, angle = -90]{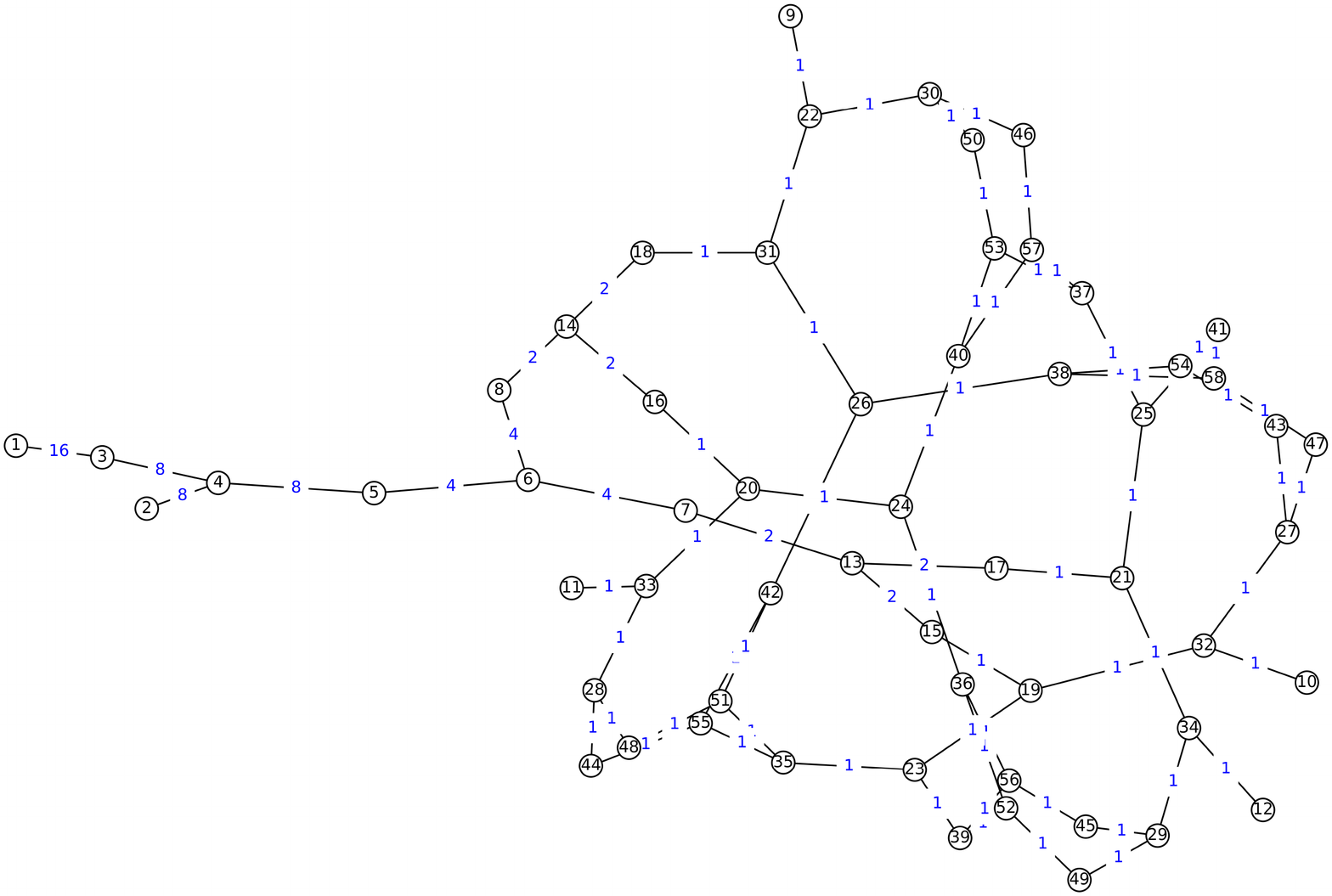}
\label{figure:dual-graph}
\end{figure}
\end{center}

\begin{rem}\rm The graph $\mathcal{G}'$ of the integral model $\mathscr{M}/\langle w_\gp\rangle$ (see Figure~\ref{figure:dual-graph})
is an example of a graph whose automorphism group does {\it not} have an element that is admissible. Indeed, it is easy to see that every
element of $\Aut(\mathcal{G}')$ must fix the vertex $v_4$ and the $3$ edges of weight $8$ contained in $\Star(v_4)$. However, the
vertex $v_4$ and $\Star(v_4)$ are removed when we blow down $\mathscr{M}/\langle w_\gp\rangle$ to obtain the stable model
$(\mathscr{M}/\langle w_\gp\rangle)_{st}$ (see Figure~\ref{figure:st-dual-graph}). This example shows that~\cite[Proposition 3.4]{kr08}
is incorrect as stated and needs to be modified slightly.  
\end{rem}

\subsection{Hyperellipticity of the curve $C$} We are now ready to prove one of our main results.

\begin{thm}\label{thm:C-is-hyperelliptic} The curve $C$ is hyperelliptic over $F$. 
\end{thm}

\begin{proof} Let $\gamma \in \Gamma_\CO$ be an elliptic element of order $2$, and $P$ a fixed point by $\gamma$.
Then $P$ is a CM point by construction, and $\gamma$ acts on the local ring $\CO_{C,P}$ as an involution. More
specifically, letting $t$ be a uniformiser at $P$, we see that $\gamma$ acts on $t$ as: 
$$t (\bmod\, t^2) \mapsto -t (\bmod\, t^2).$$
This forces any global differential form in $H^0(C, \Omega^1_C)$, which vanishes at $P$, to vanish to {\it even} order. 
We claim that this implies that $P$ is a Weierstrass point. To prove this, we use Riemann Roch.

Let $K_C$ the canonical divisor. Then, we have $\ell(K_C - 2P) = \ell(K_C - P)$, i.e. every differential that vanishes 
at $P$ vanishes to order 2. By Riemann-Roch, we  have
$$\ell(K_C - 2P) - \ell(2P) =  \deg(K_C - 2P) - g + 1 = (2g-4) - g + 1 = g - 3;$$
and
$$\ell(K_C - P) - \ell(P) =  \deg(K_C - P) - g + 1 = (2g-3) - g + 1 = g - 2.$$
So, if $\ell(K_C - 2P) = \ell(K_C - P)$, then $\ell(2P) - \ell(P) = 1$. Now, since $\mathscr{L}(P)$ is the space of
constant functions,  we see that $\mathscr{L}(2P)$ must be non-trivial, and thus $P$ is a (hyperelliptic) Weierstrass point.

From the above argument, it follows that $C$ has $17$ hyperelliptic Weierstrass points that are all CM. By
Shimura reciprocity law, these CM points are all defined over the Hilbert class field $H_K$ of $K$, 
where $K$ is the CM-field defined in Subsection~\ref{subsec:CM-extension}. Let 
$M$ be the normal closure of $H_K$ over $F$. Then $[M:F] = 34$ and $\Gal(M/F) \simeq D_{17}$; and the
action of $\Gal(\overline{F}/F)$ on the set of Weierstrass points $\mathscr{W}$ must factor through it 
(see Subsection~\ref{subsec:galois-action}). Therefore, we must have $\#\mathscr{W} = 34$. In other words, $C$ has 
$34$ hyperelliptic Weierstrass points. Since $C$ has genus $16$, it must therefore be hyperelliptic by Proposition~\ref{prop:hyperellipticity}.  
\end{proof}

\begin{rem}\rm It follows from the proof of Theorem~\ref{thm:C-is-hyperelliptic} that half of the Weierstrass points
on $C$ are CM, while the remaining half are non CM. This means that the hyperelliptic involution must necessarily
be exceptional. However, this should be expected since $C = X_0^D(1)/\langle w_D\rangle$, where $w_D$ is the 
{\it unique} Atkin-Lehner involution acting on $X_0^D(1)$. 
\end{rem}

\begin{thm}\label{thm:AutC}
The automorphism group of the curve $C$ is $\Aut(C) = \Z/2\Z$. 
\end{thm}

\begin{proof}
By Theorem~\ref{thm:C-is-hyperelliptic}, the group $\Aut(C)$ is non-trivial since it contains the hyperelliptic involution. 
By Lemma~\ref{lem:AutC-embeds}, it injects into an admissible  subgroup $H$ of $\Aut(\mathcal{G}_{st})$ of order $2$.
\end{proof}

\begin{rem}\rm By Theorem~\ref{thm:AutC}, $\Aut(C) = \Z/2\Z$, so that $\Aut(X_0^D(1)) = (\Z/2\Z)^s$, with $1\le s \le 2$. 
We note that $s = 2$ if and only if the hyperellipitc involution on $C$ comes from  an exceptional automorphism on $X_0^D(1)$. 
We also note that it is conjectured that there are only finitely many Shimura curves $X$ defined over $\Q$ such that $\Aut(X)$ 
contains an exceptional automorphism (see~\cite{kr08}). This conjecture would imply that there are very few Shimura 
curve quotients defined over $\Q$ which have automorphisms arising from exceptional automorphisms. However, analogues of this 
conjecture have barely been explored over totally real fields.
\end{rem}

\begin{thm}\label{thm:C-is-hyperelliptic-over-Q}
The curve $C$ is hyperelliptic over $\Q$. 
\end{thm}

\begin{proof} Since $C$ descends to $\Q$, it is enough to show that the hyperelliptic involution $\iota:\,  C\to C$ also descends to $\Q$. 
By Theorem~\ref{thm:AutC}, $\Aut(C)/\langle \iota\rangle$ is trivial. Furthermore, the field $\Q(\zeta_{32})$ is a splitting field for $D$ 
whose class number is one. So the CM point attached to the extension $\Q(\zeta_{32})/F$ is defined over $F$. So $C$ descends to $\Q$ 
as a hyperelliptic curve by~\cite[Proposition 4.8]{sv16}.
\end{proof}

\begin{figure}
\caption{The dual graph $\mathcal{G}_{st}$ of the stable model for the quotient $\mathscr{M}/\langle w_\gp\rangle$.}
\begin{center}
\vspace{-1.5cm}
\includegraphics[width = .80\textwidth]{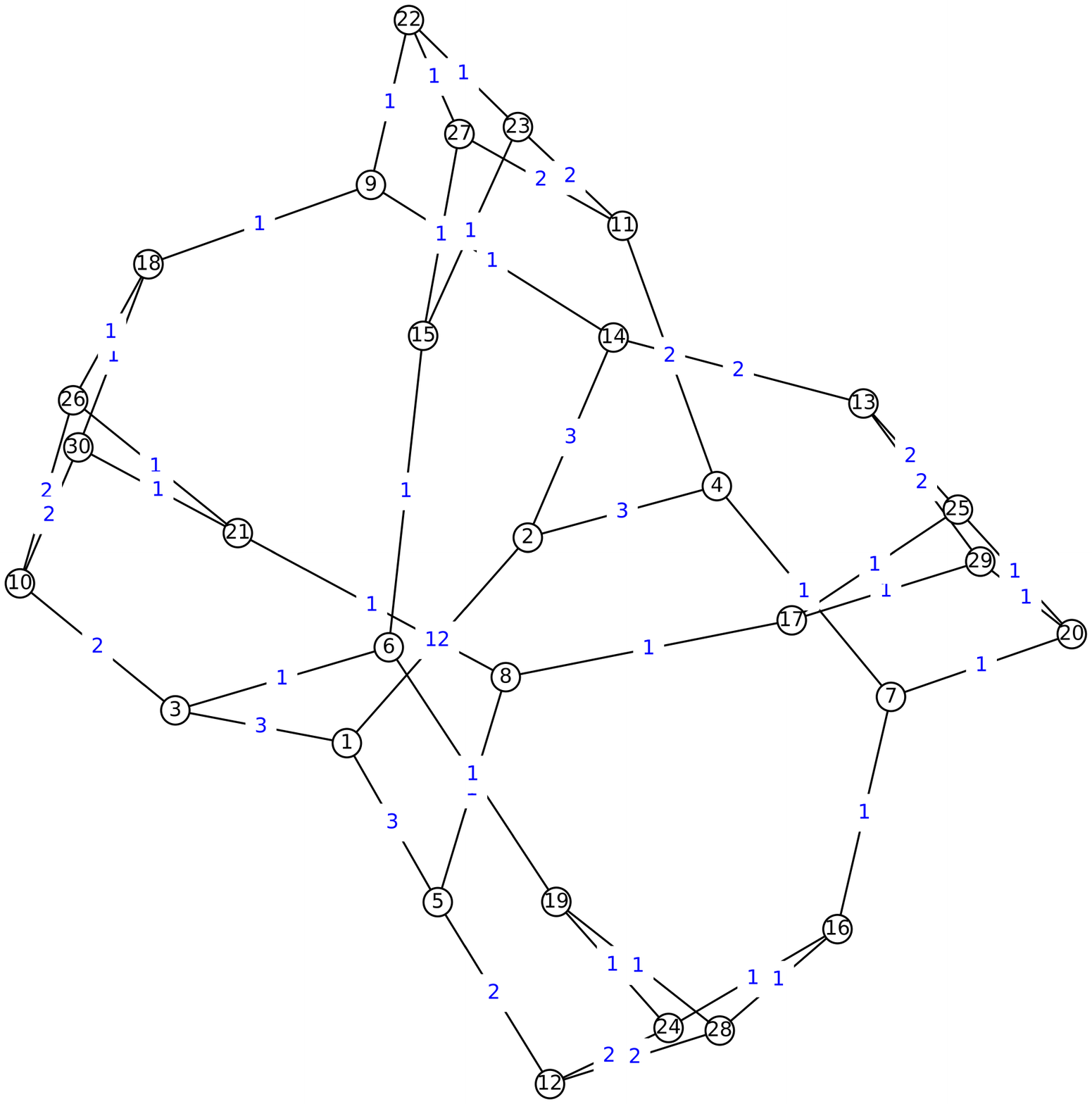}
\vspace{-2cm}
\end{center}
\label{figure:st-dual-graph}
\end{figure}

\begin{rem}\rm 
One should be able to compute an equation for $C$ by using~\cite{voight-willis14}. However, currently, the strategy for doing so is not fully 
implemented. It should also be possible to use a generalisation of the $p$-adic approach discussed in~\cite{cm14}, which was inspired 
by~\cite{kur79, kur94}. Given that the determination of equations for Shimura curves defined over totally real fields is one question that is of 
independent interest in its own right, we hope to return to this in the future. 
\end{rem}

\begin{rem}\rm  
We note that Michon~\cite{mich81} (and also unpublished work of Ogg) provides a complete list of all hyperelliptic Shimura 
curves with square-free level defined over $\Q$. Shimura curves defined over $\Q$ which admit hyperelliptic quotients have
also been investigated quite a bit, see for example~\cite{mol12, gm16, gy17} and references therein. In contrast, there has
been very little work on these types of questions for Shimura curves defined over totally real fields $F$ larger than $\Q$. This makes 
Theorem~\ref{thm:C-is-hyperelliptic} of the more striking. Indeed, not only does it give one of the few examples of Shimura curves with a 
hyperelliptic quotient over a totally real field, but also one whose genus is larger than most known examples over $\Q$. 
\end{rem}

\subsection{The Jacobian varieties $\Jac(X_0^D(1))$ and $\Jac(C)$}\label{subsec:decomposition}
In this section, we explain the connection between the simple factors of $\Jac(X_0^D(1))$ and the conjectures in~\cite{gro16}. 
There is more in~\cite{cd17}, where this connection is established via lifts of Hilbert modular forms.

From the discussion in Subsections~\ref{subsec:newforms} and~\ref{subsec:curves}, 
we have the decomposition for $\mathrm{Jac}(X_0^D(1))$ over $F$ (up to isogeny):
\begin{align}\label{eq:jac-decomp}
\mathrm{Jac}(X_0^D(1)) \sim A_f \times A_{f'} \times A_g \times A_{g'}\times A_h.
\end{align}
From~\eqref{eq:jac-decomp}, and the fact that $w_D = -w$, we see that 
\begin{align}\label{eq:jac-decomp2}
\Jac(C) \sim A_f \times A_{f'} \times A_g \times A_{g'}.
\end{align}
The fourfolds $A_f$ and $A_{f'}$ (resp. $A_{g}$ and $A_{g'}$) are Galois conjugate. 
We will see later that one of consequences of the compatibility between the base change action and 
Hecke orbits is that the decompositions~\eqref{eq:jac-decomp} and~\eqref{eq:jac-decomp2} descend to subfields of $F$.

\begin{thm}\label{thm:Ah} The abelian variety $A_h$ descends to a $24$-dimensional variety $B_h$ defined over $\Q$,
with good reduction outside $2$, such that $\End_{\Q}(B_h)\otimes\Q = K_h$ and 
$$L(B_h, s) = \prod_{\Pi' \in [\Pi_h]} L(\Pi', s),$$
where $\pi_h$ is the automorphic representation of $\GL_2(\A_F)$ attached to $h$, $\Pi_h$ it lifts to 
$\GSpin_{17}(\A_{\Q})$, and $[\Pi_h]$ the Hecke orbit of $\Pi_h$.
\end{thm}

\begin{proof}
By Table~\ref{table:eigenforms}, there exists a generator $\tau \in \Gal(L_h/K_h)$ such that  ${}^\sigma\! h = h^\tau$. 
So, by ~\cite[Theorem 5.4]{cd17}, $\pi_h$ lifts to an automorphic representation $\Pi_h$ on a 
split form of $\GSpin_{17}(\A_{\Q})$, with field of rationality the cubic field $K_h$. The Hecke orbit $[\Pi_h]$ of $\Pi_h$ 
has $3$ elements, and by functoriality 
$$L(B_h, s) = \prod_{\Pi' \in [\Pi_h]} L(\Pi', s).$$
It follows that $\End_{\Q}(B_h)\otimes\Q = K_h$. Since the level of the form $h$ is the unique prime $\gp$ above $2$,  
$B_h$ has good reduction outside $2$. 
\end{proof}

\medskip
Now, we turn to the quotient $C := X_0^D(1)/\langle w_D \rangle$.

\begin{thm}\label{thm:fourfolds}
The abelian varieties $A_f$ and $A_{f'}$ (resp. $A_{g}$ and $A_{g'}$) descend to pairwise conjugate fourfolds $B_f$ and 
$B_{f'}$ (resp. $B_{g}$ and $B_{g'}$) over $\Q(\sqrt{2})$, with trivial endomorphism rings, such that 
\begin{align*}
L(B_f, s) &= L(\Pi_f, s)\,\,\text{and}\,\, L(B_{f'}, s) = L(\Pi_{f'}, s),\\
L(B_g, s) &= L(\Pi_g, s)\,\,\text{and}\,\, L(B_{g'}, s) = L(\Pi_{g'}, s),
\end{align*}
where $\pi_f, \pi_{f'}, \pi_g$ and $\pi_{g'}$ are the automorphic representations of $\GL_2(\A_F)$ attached to
$f$, $f'$, $g$ and $g'$, respectively; and $\Pi_f, \Pi_{f'}, \Pi_g$ and $\Pi_{g'}$ their respective lifts to 
$\GSpin_{9}/\Q(\sqrt{2})$. They have good reduction outside $(\sqrt{2})$. 
\end{thm}

\begin{proof} The identities in Table~\ref{table:eigenforms}, combined with \cite[Theorem 5.4]{cd17}, implies that 
$\pi_f, \pi_{f'}, \pi_g$ and $\pi_{g'}$ indeed lift to automorphic representations $\Pi_f, \Pi_{f'}, \Pi_g$ and $\Pi_{g'}$ on 
$\GSpin_9/\Q(\sqrt{2})$ with field of rationality $\Q$. Consequently, the fourfolds $A_f$, $A_{f'}$, $A_g$ and $A_{g'}$ 
descend to pairwise conjugate fourfolds $B_f$ and $B_{f'}$ (resp. $B_g$ and $B_{g'}$) such that
$$\End_{\Q(\sqrt{2})}(B_f) = \End_{\Q(\sqrt{2})}(B_{f'}) = \End_{\Q(\sqrt{2})}(B_g) = \End_{\Q(\sqrt{2})}(B_{g'}) = \Z.$$
The equalities of $L$-series follow by functoriality. For the same reason as above, the fourfolds have good reduction outside $(\sqrt{2})$. 
\end{proof}

\begin{rem}\rm The decomposition~\eqref{eq:jac-decomp} is only true {\it a priori} over $F$. 
However, Theorem~\ref{thm:Ah} and Theorem~\ref{thm:fourfolds} imply that it descends to $\Q(\sqrt{2})$. In fact, the
products $A_f \times A_{f'}$ (resp. $A_g \times A_{g'}$) further descend to $\Q$. And so, the decomposition~\eqref{eq:jac-decomp}
will descend to $\Q$ if we group them accordingly. 
\end{rem}

\subsection{The connectedness of $\Spec(\T)$} 
Let $\T$ be the $\Z$-subalgebra of $\End_\C(S_{2}^D(1))$ acting on $S_2^D(1)$. We recall that $S_2^D(1)$ is isomorphic
to $S_2(\gp)^{\rm new}$ as a Hecke module.  

\begin{prop}\label{prop:hecke-connectedness} $\Spec(\T)$ is connected.
\end{prop}

\begin{proof} The curve $X_0^D(1)$ is a Shimura curve of prime level, and each Hecke constituent appears with multiplicity one.
So, the proof in~\cite[Proposition 10.6]{maz77} applies readily. 
\end{proof}

The following two propositions determine the congruences which realise the connectedness of $\Spec(\T)$.

\begin{prop} 
The forms $f,$ $f',$ $g$ and $g'$ are congruent modulo $5$.
\end{prop}

\begin{proof} The prime $5$ is totally ramified in $L_f = L_{f'}$. Let $\gP_5$ be the unique prime above it, and 
$\rho_{f, 5}, \rho_{f', 5}: \, \Gal(\Qbar/F) \to \GL_2(\CO_{L_f, \gP_5})$ the $\gP_5$-adic representations attached
to $f$ and $f'$, respectively. By reduction modulo $\gP_5$, we get two representations $\bar{\rho}_{f, 5}, \bar{\rho}_{f', 5}: \, \Gal(\Qbar/F) \to \GL_2(\F_5)$. 
From Table~\ref{table:eigenforms}, we have and ${}^\sigma\!f = f'$, ${}^{\sigma^2}\!f = f^\tau$. Also, since $\gP_5$ is totally ramified
in $L_f$, we have $\tau(\gP_5) = \gP_5$. It follows that $\bar{\rho}_{f, 5} = \bar{\rho}_{f', 5}$ is a base change from $\Q(\sqrt{2})$. 
The prime $5$ is also totally ramified in $L_g = L_{g'}$. With obvious notations, the same argument as above shows that 
$\bar{\rho}_{g, 5} = \bar{\rho}_{g', 5}$ is also a base change from $\Q(\sqrt{2})$.

By using the multiplicity one argument in~\cite[\S 6]{bcddf18}, we show that $\bar{\rho}_{f, 5} \simeq \bar{\rho}_{g, 5}$. This implies that $f$, $f'$, $g$ 
and $g'$ are congruent modulo $5$. 
\end{proof}

\begin{prop} 
The forms $f,$ $f'$ and $h$ are congruent modulo $3$.
\end{prop}

\begin{proof} There is a unique prime $\gP_3$ above $3$ in $L_f = L_{f'}$; it has inertia degree $2$ and ramification degree $2$. 
Let $\rho_{f, 3}, \rho_{f', 3}: \, \Gal(\Qbar/F) \to \GL_2(\CO_{L_f, \gP_3})$ the $\gP_3$-adic representations attached to $f$ and $f'$, 
respectively. By reduction modulo $\gP_3$, we get two representations $\bar{\rho}_{f, 3}, \bar{\rho}_{f', 3}: \, \Gal(\Qbar/F) \to \GL_2(\F_9)$. 
From Table~\ref{table:eigenforms}, we have and ${}^\sigma\!f = f'$, ${}^{\sigma^2}\!f = f^\tau$. Also, since $\gP_3$ is the unique prime
above $3$ in $L_f$, we have $\tau(\gP_3) = \gP_3$. It follows that $\bar{\rho}_{f, 3} = \bar{\rho}_{f', 3}$ is a base change from $\Q(\zeta_{16})^+$.

In the cubic subfield $K_h$ of $L_h$, the prime $3$ factors as $(3) = \gp_3 \gp_3'$, where $\gp_3$ has inertia degree $1$, and $\gp_3'$ inertia
degree $2$. The prime $\gp_3'$ is totally ramified in $L_h$. We let $\gP_3'$ be the unique prime above it, and 
$\rho_{h, 3}:\,\Gal(\Qbar/F) \to \GL_2(\CO_{L_h, \gP_3'})$ the $\gP_3'$-adic representation attached to $h$.  
By reduction modulo $\gP_3'$, we get a representation $\bar{\rho}_{h, 3}: \, \Gal(\Qbar/F) \to \GL_2(\F_9)$.  From Table~\ref{table:eigenforms}, 
we have and ${}^\sigma\!h = h^\tau$. Also, since $\gP_3'$ is the unique prime above $\gp_3'$ in $L_h$, we have $\tau(\gP_3') = \gP_3'$. 
It follows that $\bar{\rho}_{h, 3}$ is also a base change from $\Q(\zeta_{16})^+$.

By using the multiplicity one argument in~\cite[\S6]{bcddf18}, we show that $\bar{\rho}_{f, 3} \simeq \bar{\rho}_{h, 3}$. This implies that $f$, $f'$
and $h$ are congruent modulo $3$. 
\end{proof}

\section{\bf The $2$-torsion field of $\Jac(X_0^D(1))$ and the Harbater field}\label{sec:2-torsion-fields}

The main result of this section establishes that every simple factor of $\Jac(X_0^D(1))$ has a $2$-torsion field whose
normal closure is the Harbater field. We start with the following theorem. 

\begin{thm}\label{thm:harbater-field} Let $N$ the field of $2$-torsion of $\Jac(C)$ over $\Q$. Then $N$ is the Harbater field.
\end{thm}

\begin{proof} Keeping the notation in the proof of Theorem~\ref{thm:C-is-hyperelliptic}, $N$
is the normal closure of $M$. It follows from this, and direct calculations, that $\Gal(N/\Q) \simeq F_{17}$.
By construction, $N$ is unramified outside $2$ and $\infty$. However, by~\cite[Theorem 2.25]{har94}, there is a unique
Galois number field unramified outside $2$ and $\infty$, with Galois group $F_{17}$. So $N$ must be the Harbater field.
\end{proof}

\begin{rem}\rm The field $N$ is the splitting field of the polynomial
\begin{align*}
H &:= x^{17} - 2x^{16} + 8x^{13} + 16x^{12} - 16x^{11} + 64x^9 - 32x^8 - 80x^7 + 32x^6 + 40x^5\\
&\qquad{} + 80x^4 + 16x^3 - 128x^2 - 2x + 68. 
\end{align*}
This polynomial was computed by Noam Elkies following a \verb|mathoverflow.net|~\cite{er15} discussion initiated by Jeremy Rouse.
We thank David P. Roberts for bringing this discussion to our attention. 
\end{rem}

\subsection{The mod $2$ Hecke eigensystems} 
Let $\T_f$, $\T_{f'}$ $\T_g$, $\T_{g'}$ and $\T_h$ be the $\Z$-subalgebras acting 
on the Hecke constituents of $f$, $f'$, $g$, $g'$ and $h$ respectively. From the discussion in Subsection~\ref{subsec:newforms}, we have 
\begin{align*}
\T\otimes\Q &= (\T_f \otimes \Q) \times (\T_{f'} \otimes \Q) \times (\T_g \otimes \Q) \times (\T_{g'} \otimes \Q)
 \times (\T_h \otimes \Q)\\
&= L_f \times L_{f'} \times L_{g} \times L_{g'} \times L_{h}.
\end{align*}

By direct calculations, we get the following:
\begin{itemize}
\item $[\CO_{L_f} : \T_f] = [\CO_{L_{f'}} : \T_{f'}]$ divides $3$,
\item $[\CO_{L_g} : \T_g]  = [\CO_{L_{g'}} : \T_{g'}] = 1$,
\item $[\CO_{L_h} : \T_h]$ divides $3\cdot 5^6$.
\end{itemize}
Therefore $\T \otimes \Z_2$ decomposes into $\Z_2$-algebras as
$$\T \otimes \Z_2 = (\T_f \otimes \Z_2) \times (\T_{f'} \otimes \Z_2) \times (\T_g \otimes \Z_2) \times (\T_{g'} \otimes \Z_2)
\times (\T_h \otimes \Z_2).$$
The prime $2$ is inert in $L_f = L_{f'}$, and $L_g = L_{g'}$, so the first four factors are local $\Z_2$-algebras. 
Let $\gm_f$, $\gm_{f'}$, $\gm_g$ and $\gm_{g'}$ be the corresponding maximal ideals. Then, by the identities in Table~\ref{table:eigenforms}, 
we have $\sigma(\gm_f) = \gm_{f'}$ and $\sigma^2(\gm_f) = \tau_f(\gm_f)$ for some $\tau_f \in \Gal(\F_{16}/\F_2)$; and $\sigma(\gm_g) = \gm_{g'}$ 
and $\sigma^2(\gm_g) = \tau_g(\gm_g)$ for some $\tau_g \in \Gal(\F_{16}/\F_2)$. We let $\theta_f, \theta_{f'}, \theta_g, \theta_{g'}:\,\T\otimes \Z_2 \to \F_{16}$ 
be  the corresponding mod $2$ Hecke eigensystems.

Next, we recall that $L_h$ is the ray class field of conductor $\gc = (\frac{1}{2}(c^2 - 16c + 25))$ over the field $K_h = \Q(c)$, with $c^3 + c^2 - 229c + 167 = 0.$
The prime $2$ is totally ramified in $K_h$. Letting $\gp_2$ be the unique prime above it, we get that $\gp_2 = \gP \gP'$, where $\gP$ and $\gP'$ are inert primes,
and $\tau(\gP) = \gP'$. Therefore, there are two maximal ideals $\gm_{h}$ and $\gm_{h}'$ in $\T_h \otimes \Z_2$ such that 
$\sigma(\gm_{h}) = \gm_{h}'$ and $\sigma^2(\gm_{h}) = \tau_h(\gm_{h})$. We let $\theta_h, \theta_h':\,\T\otimes \Z_2 \to \F_{16}$ be
the resulting two mod $2$ Hecke eigensystems. 
 
\begin{prop}\label{prop:mod-2-constituents} The forms $f$, $f'$, $g$, $g'$ and $h$ give rise to two mod $2$ Hecke eigensystems $\theta$ and $\theta'$ that
$\theta' = \theta \circ \sigma$ and $\theta \circ \sigma^2 = \bar{\tau} \circ \theta$, where $\Gal(\F_{16}/\F_2)=\langle \bar{\tau}\rangle$. Up to relabelling, 
we have $\theta = \theta_f = \theta_g = \theta_h$, and $\theta' = \theta_{f'} = \theta_{g'} = \theta_{h}'$.
\end{prop}

\begin{proof}
We will apply the multiplicity one argument in~\cite[\S 6]{bcddf18} to deduce that, up to relabelling, $\theta_f = \theta_g = \theta_h$, and 
$\theta_{f'} = \theta_{g'} = \theta_{h}'$. Let $M$ be the underlying $\F_2$-module to $\T \otimes \F_2$. Then, the pair $(\theta_f,\theta_{f'})$
comes from two simple Hecke constituents of dimension $4$ over $\F_2$ that are conjugate by the action of $\Gal(F/\Q)$. These Hecke constituents
belong to the socle $S$ of $M$, i.e. the largest semi-simple $\T\otimes\F_2$-submodule of $M$. Likewise for the pairs $(\theta_g,\theta_{g'})$ and
$(\theta_h,\theta_{h}')$. Let $\T'$ be the $\Z$-subalgebra of $\T$ generated by the Hecke operators $T_\gp$, with $\N\gp \le 1000$.
We view $M$ as a $\T'\otimes \F_2$-module, and let $S'$ be its socle. By direct calculations in \verb|magma|, we show that $S'$ has two irreducible constituents,
and each constituent has dimension $4$ and multiplicity one. Furthermore, each of those constituents decomposes into $4$ one-dimensional 
Hecke constituents over $\T'\otimes \F_{16}$. This means that $S'$ must necessarily be the socle of $M$ viewed as a $\T\otimes \F_2$-module, and that
its $\T'\otimes \F_{16}$-decomposition is also the $\T\otimes \F_{16}$-decomposition of $S$. By comparing these one-dimensional $\F_{16}$-valued
Hecke eigensystems with the reduction modulo $2$ of the Hecke eigenvalues of the newforms in $S_2^D(1)$, we see that $\theta_f = \theta_g = \theta_h$, 
and $\theta_{f'} = \theta_{g'} = \theta_{h}'$, up to relabelling. The identities $\theta' = \theta \circ \sigma$ and $\theta \circ \sigma^2 = \bar{\tau} \circ \theta$
follow from the relations between the forms.
\end{proof}

\subsection{The fields of $2$-torsion of the simple factors of $\Jac(X_0^D(1))$}

\begin{thm}\label{thm:2-torsion-fields}
Let $A$ be a simple factor of $\Jac(X_0^D(1))$, and $M = F(A[2])$ the field of $2$-torsion of $A$.
Then, the normal closure of $M$ is the Harbater field $N$.
\end{thm}

We will give two proofs of this result, starting with the simplest one.

\begin{proof}[First proof of Theorem~\ref{thm:2-torsion-fields}] In light of Proposition~\ref{prop:mod-2-constituents},
it is enough to prove this for the simple factors of $\Jac(C)$. To this end, recall that
\begin{align*}
\Jac(C) \sim A_f \times A_{f'} \times A_g \times A_{g'}.
\end{align*}
By Theorem~\ref{thm:harbater-field}, we know that the field of $2$-torsion of $\Jac(C)$ is the Harbater field.
So it is enough to show that the compositum of the fields of $2$-torsion of its simple factors is also the Harbater field.
But, again by Proposition~\ref{prop:mod-2-constituents}, $A_f$ and $A_g$ have the same field of $2$-torsion.
It is the field $M_\theta$ cut out by the Galois representation attached to the Hecke eigensystem $\theta$. Similarly,
$A_{f'}$ and $A_{g'}$ have the same field of $2$-torsion, the field $M_{\theta'}$ cut out by the Galois representation 
attached to $\theta'$. Since $\theta$ and $\theta'$ are interchanged by $\Gal(F/\Q)$, we must have 
$M_{\theta'} = M_\theta^\sigma$. Therefore $M_\theta$ and $M_{\theta'}$ have the same normal closure $N_\theta = N_{\theta'}$.
By replacing the isogeny $\phi_f:\,\Jac(C) \to A_f$ if necessary, we can assume that $M_\theta$, and hence $N_\theta$,
is a subfield of $N$. From the Frobenius data attached to $\theta$, we see that the order of $\Gal(N_\theta/\Q)$ is divisible by
$17$, hence $[N:N_\theta] \mid 16$. Since $\Gal(N/\Q) \simeq F_{17}$ has no non-trivial normal 
subgroup whose order divides $16$, we conclude that $N_\theta = N$.
\end{proof}

For the second proof, we need the following result. 

\begin{prop}\label{prop:dihedral-rep} 
Let $\theta$ and $\theta'$ be the Hecke eigensystems in Proposition~\ref{prop:mod-2-constituents}, and 
$\bar{\rho},\, \bar{\rho}': \Gal(\Qbar/F) \to \GL_2(\F_{16})$ the mod $2$ Galois representations attached to them.
Then, there are characters $\chi, \chi': \Gal(\Qbar/K) \to \F_{2^8}^\times$, with trivial conductor such that $\bar{\rho} = \Ind_{K}^F \chi$, and 
$\bar{\rho}' = \Ind_{K}^F \chi'$. 
\end{prop}

\begin{proof} We already computed the Hecke constituents of the space $S_2(1)$ in~\cite{dem09}. The mod $2$ Hecke eigensystems in that case
have coefficient fields $\F_{2^s}$ where $s = 1, 2, 8$. Therefore, since $\theta$ has coefficient field $\F_{16}$, it cannot arise from an eigenform of level $1$. 
By the Serre conjecture for totally real fields (the totally ramified case)~\cite{gs11}, it must appear on the quaternion algebra $D'$ with level $(1)$ and non-trivial 
weight. The same is true for $\theta'$. In fact, the analysis conducted in the proof of Proposition~\ref{prop:mod-2-constituents} also shows that they are the only 
eigensystems that can appear at that weight. (We note that there are only two Serre weights in this case.)

\medskip
Let $\chi:\, \Gal(\Qbar/K) \to \overline{\F}_2^\times$ be a character with trivial conductor such that $\chi^s \neq \chi$, where $\Gal(K/F) = \langle s \rangle$. 
By class field theory, we can identity $\chi$ with its image under the Artin map. Since $\chi$ is unramified, it must factor as 
$\chi: K^\times \backslash \A_K^\times \twoheadrightarrow \Cl_K \to \overline{\F}_2^\times$.
Furthermore, since $\Cl_K \simeq \Z/17\Z$, we must have 
$\chi: K^\times \backslash \A_K^\times \to \F_{2^8}^\times$, and the representation $\bar{\rho}_{\chi}:=\Ind_K^F \chi:\, \Gal(\Qbar/F) \to \GL_2(\F_{16})$ has 
coefficients in $\F_{16}$. So, $\bar{\rho}_{\chi}$ has level $(1)$ and non-trivial weight by the argument above. Therefore, it must be isomorphic to a 
Galois conjugate of $\bar{\rho}$. Up to relabelling, we can assume that 
$\bar{\rho} \simeq \bar{\rho}_{\chi}$. Since $\theta$ and $\theta'$ are $\Gal(F/\Q)$-conjugate, there is also a character 
$\chi': K^\times \backslash \A_K^\times \to \F_{2^8}^\times$ such that $\bar{\rho}' \simeq \bar{\rho}_{\chi'}$. 

\medskip
Alternatively, we can show that $\theta$ appears on $D'$ with the non-trivial weight without
using the fact that it has coefficients in $\F_{16}$. Indeed, we have 
$$\bar{\rho}_{\chi}|_{I_\gp} \simeq \begin{pmatrix} 1& \ast\\ 0 & 1\end{pmatrix}.$$
Let $K_{\gP}$ be the completion of $K$ at $\gP$, the unique prime above $\gp$. 
Since $K = F[\beta]$, and $\beta^2 = -2 - \alpha$ is a generator of $\gp$, then we have $K_{\gp} = F_\gp[\sqrt{\varpi}]$,
where $\varpi$ is a uniformiser of $F_\gp$. Therefore, $\bar{\rho}_{\chi}|_{D_\gp}$ doesn't arise from a finite flat group 
scheme. Hence, $\bar{\rho}_{\chi}$ must have non-trivial weight. 
\end{proof}

We are now ready for the second proof of Theorem~\ref{thm:2-torsion-fields}. 

\begin{proof}[Second proof of Theorem~\ref{thm:2-torsion-fields}] 
Let $\bar{\rho}_{\theta},\,\bar{\rho}_{\theta'}:\,\Gal(\Qbar/F) \to \GL_2(\F_{16})$ be the mod $2$ Galois representations attached to the
eigensystems $\theta$ and $\theta'$. By Proposition~\ref{prop:dihedral-rep}, $\bar{\rho}_{\theta}$ and $\bar{\rho}_{\theta'}$ are dihedral and we have that 
$\mathrm{im}(\bar{\rho}_{\theta}) = \mathrm{im}(\bar{\rho}_{\theta'}) = D_{17}$. Let $M_\theta, M_{\theta'}$ be the fields cut out by $\bar{\rho}_{\theta}$ 
and $\bar{\rho}_{\theta'}$; and $N_\theta$ and $N_{\theta'}$ the normal closure of $M_\theta$ and $M_{\theta'}$, respectively. 
By Proposition~\ref{prop:mod-2-constituents}, we have $M_{\theta'} = M_{\theta}^\sigma$, hence $N_\theta = N_{\theta'}$.  Also, by construction $M_\theta$ 
and $M_{\theta}^\sigma$ are unramified extension of $K$. So, by uniqueness of the Hilbert class field, we must have 
$M_\theta = M_{\theta}^\sigma = M_\theta M_{\theta}^\sigma = H_K$, where $M_\theta M_{\theta}^\sigma$ is the compositum of $M_\theta$ and 
$M_{\theta}^\sigma$; and $H_K$ is the Hilbert class field of $K$. Since $\theta \circ \sigma^2 = \bar{\tau} \circ \theta$, we have  
$$\Gal(N_\theta/\Q) = D_{17}\rtimes \Z/8\Z = F_{17}.$$
Again by~\cite[Theorem 2.25]{har94}, we must have $N = N_\theta = N_{\theta'}$.
\end{proof}

\begin{rem}\rm
From Theorem~\ref{thm:2-torsion-fields}, we see that none of the fourfolds $A_f$, $A_{f'}$, $A_g$ or $A_{g'}$ can be the 
Jacobian of a hyperelliptic curve since the action of $\Gal(\Qbar/F)$ on the points of $2$-torsion cannot factor through $S_{10}$ 
(see Subection~\ref{subsec:galois-action}). However, as we explained earlier, $A_f$, $A_{f'}$, $A_g$ and $A_{g'}$ descend, separately, 
into pairwise conjugate abelian varieties over $\Q(\sqrt{2})$. And the products $A_f \times A_{f'}$ and $A_g \times A_{g'}$ are $8$-dimensional
abelian varieties which further descend to $\Q$. So, we conclude with the following questions. 
Do there exist hyperelliptic curves $C_f$ and $C_g$ defined over $F$ such that, we have
$$\Jac(C_f) \sim A_f \times A_{f'},\,\text{and}\,\, \Jac(C_g) \sim A_g \times A_{g'}?$$
If so, do these two curves descend to $\Q$ as well? We were asked these two questions by Noam Elkies in an email. An affirmative answer to 
them would mean that the Harbater field is given by hyperelliptic curves of genus $8$, which is much smaller.  A priori, the hyperelliptic
polynomials of these curves should have degree $18$. However, the structure of the Galois group $\Gal(N/\Q) = F_{17}$ indicates that one of their
roots would be rational and could be moved to $\infty$. This means that the hyperelliptic polynomials of the curves $C_f$ and $C_g$ would in fact
have degree $17$, the same as that of the Elkies polynomial displayed earlier.
\end{rem}

\newcommand{\etalchar}[1]{$^{#1}$}
\providecommand{\bysame}{\leavevmode\hbox to3em{\hrulefill}\thinspace}
\providecommand{\MR}{\relax\ifhmode\unskip\space\fi MR }
\providecommand{\MRhref}[2]{%
  \href{http://www.ams.org/mathscinet-getitem?mr=#1}{#2}
}
\providecommand{\href}[2]{#2}

\end{document}